\documentclass[12pt]{amsart}

\usepackage{amssymb}
\usepackage{latexsym}
\usepackage{amscd} 
\usepackage{amsmath}
\usepackage{epsfig}
\usepackage{mathrsfs}
\usepackage[margin=3.25cm]{geometry}
\usepackage{ifpdf}
\usepackage[textures,backref,bookmarks=false]{hyperref}
\usepackage{xr}
\usepackage{bm}

\usepackage{amsthm}
\newtheorem*{thmA}{Theorem~A}
\newtheorem*{thmB}{Theorem~B}
\newtheorem*{thmC}{Theorem~C}

\newtheorem*{CorD}{Corollary~D}

\newtheorem{theorem}{Theorem}[section]

\newtheorem{corollary}[theorem]{Corollary}

\newtheorem{lemma}[theorem]{Lemma}
\newtheorem{definition}[theorem]{Definition}
\newtheorem{proposition}[theorem]{Proposition}
\newtheorem{remark}[theorem]{Remark}

\newtheorem{example}[theorem]{Example}
\newtheorem{qst}[theorem]{Question}

\newtheorem*{definition*}{Definition}

\def\A{\mathbb{A}}

\def\C{\mathbb{C}}

\def\H{{\mathbb{H}}}

\def\R{\mathbb{R}}
\def\P{\mathbb{P}}

\def\N{\mathbb{N}}

\def\CP2{{\mathbb{CP}^2}}

\def\cA{{\mathcal{A}}}

\def\cF{{\mathcal{F}}}

\def\cH{{\mathcal{H}}}

\def\cO{{\mathcal{O}}}

\def\cT{{\mathcal{T}}}

\def\e{\varepsilon}

\def\aoneank{\A^{1,an}_k}
\def\poneank{\P^{1,an}_k}

\def\diam{\operatorname{diam}}
\def\logdiam{\operatorname{\log diam}}
\def\interior{\operatorname{int}}
\def\id{\operatorname{id}}

\def\unsepdeg{\operatorname{deg}^{un}}

\newcommand{\ov}{\overline}

\def\car{\operatorname{char}}
\def\val{\operatorname{val}}
\def\gal{\operatorname{Gal}}

\begin{document}
\title{A non-archimedean Montel's theorem}

\author{Charles Favre}
\address{
Centre de Math\'ematiques Laurent Schwartz
\'Ecole Polytechnique
91128 Palaiseau Cedex France 
}
\email{favre@math.polytechnique.fr}

\author{Jan Kiwi}
\address{Facultad de Matem\'aticas,
 Pontificia Universidad Cat\'olica de Chile,
 Casilla 306, Correo 22, Santiago,
 Chile.}
\email{jkiwi@puc.cl}

\author {Eugenio Trucco}
\address{Instituto de Ciencias F\'isicas y Matem\'aticas, Facultad de Ciencias, Universidad Austral de Chile, Casilla 567, Valdivia, Chile.}
\email{etrucco@uach.cl}

\begin{abstract}
We prove a version of Montel's theorem for analytic functions over a non-archime\-dean complete valued field.  We propose a definition of normal family in this context, and  give applications of our results to the dynamics of non-archimedean entire functions.
\end{abstract}

\date{\today}

\thanks{Supported by the following projects: ECOS C07E01, ANR-Berko, Research Network on Low Dimensional Dynamics ACT-17, Conicyt, Chile,  and MathAMSUD-DySET}

\maketitle

\setcounter{tocdepth}{3}
\tableofcontents

\section*{Introduction}

Montel's theorem states that any family of holomorphic maps with values in $\P^1(\C) \setminus \{  0, 1, \infty\}$ is a normal family. In particular,  one can extract subsequences 
that converge in the 
topology of the uniform convergence on compact subsets.
This result was proven at the beginning of the 20th century and soon became a landmark in complex analysis in one variable. Shortly after its publication, it was used by Fatou and Julia to set the foundations of  complex dynamics.  We refer to the survey of Zalcman~\cite{zalcman:survey} for interesting results related to normal families and applications.  
   
 Our goal is to prove a Montel's type theorem in the context of non-archimedean analysis.
More specifically we fix a complete (non trivially) valued field $(k, |\cdot|)$, and 
consider maps between open subsets of the projective line $\poneank$ over $k$ in the sense of Berkovich. The first observation is that  the obvious generalization of Montel's theorem is not  true over a non-archimedean field. In fact, any sequence of constant functions $\zeta_n\in k$ such that $|\zeta_n| = 1$ and all residues classes are distinct admits no  subsequence converging to an analytic function.

On the other hand, Hsia~\cite{hsia:montel} (see also~\cite{hu-yang}) obtained a version of Montel's theorem in the case the source space is a ball, and the target space is $\P^1_k \setminus \{  0, \infty\}$. 
Two remarks are in order about this result. First, the conclusion is that a suitable family of analytic functions is  equicontinuous.
But this does not imply the existence of convergent subsequences.
Second, the assumption on the source space is a very strong one over a non-archimedean field. For instance, Hsia's theorem fails on annuli. 
Our main theorems are attempts to remedy  these issues.

Let us mention immediately that our results rely in a crucial manner on the sequential compactness of the closed unit ball 
in the affine spaces $\mathbb{A}^{N,\mathrm{an}}_k$ for \emph{all} dimensions $N\ge1$. It was proved by the first author~\cite{favre:sequential} for some specific classes of non-archimedean fields, and by J.~Poineau~\cite{poineau:sequential} in full generality.
\begin{thmA}
Suppose $k$ is a non-archimedean complete non-trivially valued field.
Let $X$ be any  connected open subset of $\poneank$, and let $f_n : X \to 
\poneank \setminus \{0, \infty \}$ be a sequence of analytic maps.
Then there exists a subsequence $\{ f_{n_j} \}$ which converges pointwise to a map $f: X \to \poneank$. 
\end{thmA}
In fact this result is also true for any affinoid domain, but we stick to open sets for simplicity.
 
Of course there is a price to pay for this statement to be true. In general,   $f$ needs not be
analytic, nor even continuous, see Section~\ref{sec:example} for some examples. Our next result gives some basic information about  the pointwise limit of analytic maps avoiding \emph{three} points. Observe that in the previous theorem $f_n(X)$ is only assumed to avoid two points, just like in Hsia's version of Montel's theorem.

Recall that the local degree $\deg_x f$ of an analytic map at a rigid point $x\in k$
is the ramification degree of $f$ at $x$. This function extends in a natural way to the Berkovich
space, see~\cite{baker-rumely:book,faber:ramification,favre-letelier:equi}.
For a type II point\footnote{see Section \ref{sec:basics} for a formal definition} $x\in \poneank$, the local action of $f$ is encoded in
a degree $\deg_x f$ rational map $T_x f$ 
acting on the projective line over the residual field $\tilde{k}$. 
In order to state our next result we introduce the notion of unseparable degree $\unsepdeg_x(f)$ at a type II point $x$.
If the characteristic of $\tilde{k}$ is $p >0$, we have that $T_x f (z) = R(z^{p^n})$ for some separable rational map
$R$ and some $n \ge 0$. In this case we say that $\unsepdeg_x (f) = p^n$.
If the characteristic of $\tilde{k}$ is zero, then we set $\unsepdeg_x (f)=1$.

We let $\H$ be the complement of rigid points in $\poneank$.
\begin{thmB}
Let $X$ be any  connected open subset of $\poneank$. Let $f_n : X \to 
\poneank \setminus \{0, 1,\infty \}$ be a sequence of analytic maps converging pointwise to a map $f: X \to \poneank$. Suppose that for any type II point $x\in X$, the sequence $\unsepdeg_x f_n$ is bounded.

Then  the map $f$ is continuous. Moreover it is either  constant or  it maps non rigid points to $\H$ and $f(X) \subset \poneank \setminus \{0, 1, \infty \}$.
\end{thmB}
It is likely that the assumption on the inseparable degree is superfluous in the case $\car(k) =0$.

\medskip

Let us add a word about the proofs of these results (in the case maps avoid three points).
Over the complex numbers, Montel's theorem follows from the existence of a hyperbolic metric
on $\P^{1}(\C) \setminus \{0, 1, \infty \}$ that is necessarily contracted by any holomorphic map. Ultimately it relies on the fact that the universal cover of $\P^{1}(\C) \setminus \{0, 1, \infty \}$ is the unit disk (note that $\P^{1,an}_\C = \P^{1}(\C)$).
Over a non-archimedean field $k$ the space $\poneank \setminus \{0, 1, \infty \}$ is already simply connected so that the former approach fails in this context. On the other hand we may exploit
the tree structure of $\poneank$. If $X$ is a ball or an annulus, and $f_n : X \to \poneank \setminus \{0, 1, \infty\}$  is a sequence of analytic maps, then their images should avoid the convex hull of $\{0, 1, \infty\}$ which looks like a tripod. By extracting a subsequence we can reduce our analysis to the case in which all images lie in a fixed ball of bounded radius. At this point, we need to split our analysis into two cases according to whether or not the local degrees are uniformly bounded. 
When it is unbounded, the proofs of Theorems~A and~B follow by looking closely at the preimage of the center of the tripod (the Gauss point in $\poneank$).
When the local degrees are bounded, then we are essentially in the situation of a family of
polynomials $P_n$ of a fixed degree $d$ with coefficients $(a^{(n)}_0, ... ,a^{(n)}_d)$ that  are uniformly bounded. By sequential compactness, we can assume
$(a^{(n)}_0, ... ,a^{(n)}_d)$ is converging in $\mathbb{A}^{d+1,an}$ and
we show that this implies the pointwise convergence of $P_n$.

\medskip

From our results, naturally arises the question of finding a characterization for the limit maps of analytic functions. Over the complex numbers pointwise limits of analytic functions are characterized as being functions that are analytic outside a polar set. We refer to~\cite{krantz:limit} for a recent survey on this question. We shall not touch upon this problem in the present article.

\medskip

As we mentioned above, Montel's theorem in complex analysis is closely related to the notion of normal families. In a non-archimedean context, we propose the following definition.
\begin{definition*}
Let $X$ be any open subset of $\poneank$.
A family $\cF$ of analytic  functions  on $X$ with values in $\poneank$ is {\sf normal} if for any sequence $f_n \in \cF$
and any point $x \in X$, there exists a neighborhood $V \ni x$, and  a subsequence $f_{n_j}$ that is converging pointwise on $V$ to a continuous function.
\end{definition*}
Let us insist on the fact that the condition on the limit to be continuous is crucial to obtain a reasonable notion.

Recall that a normal family of analytic maps in the sense of~\cite[Definition 5.38]{hu-yang} 
is a set of functions that are equicontinuous at any rigid points with respect to the chordal metric $d(\cdot, \cdot)$ on the standard projective line, where  $d(z,w) = |z-w|/(\max \{ 1, |z|\} \, \max \{ 1, |w| \})$. These two notions of normality are related 
as follows. \begin{thmC}
Let $X\subset \poneank$ be any open subset, and let $\cF$ be a family of analytic functions on $X$ with values in $\poneank$. Then the following statements are equivalent:
\begin{itemize}
\item
$\cF$ is normal in a neighborhood of any rigid point;
\item
$\cF$ is equicontinuous at any rigid point with respect to the chordal metric on $\P^1(k)$.
\end{itemize}
\end{thmC}
These three theorems subsequently imply
\begin{CorD}
Any family of meromorphic functions on an open subset $X$ of $\poneank$ such 
that, for all $x \in X$,  local unseparable degrees at $x$ are bounded, and
avoids three points in $\poneank$ is both normal, and equicontinuous at any rigid point.
\end{CorD}
We give two dynamical applications of this fact. First we prove that the domain of normality
of a rational map coincides with its Fatou set, see Theorem~\ref{thm:normal} below. Recall that
Rivera-Letelier proved that the Fatou set coincides with the equicontinuity locus (for the uniform structure) in the case  $k = \C_p$, ~\cite[Theorem 10.72]{baker-rumely:book}.  But
no characterization of the Fatou set in terms of  equicontinuity properties of the sequence of iterates was previously known in full generality. We refer to~\cite[pp.334--335]{baker-rumely:book} for  an interesting discussion on this problem.

As a second application, we define the Julia set of an entire function in $\aoneank$
as the complement of its domain of normality. In a sense, we put the work of Bezivin~\cite{bezivin:entire}
in the framework of Berkovich spaces. We extend  his work by showing that periodic 
orbits are dense in the Julia set at least when $\car(\tilde{k}) =0$. We also show that contrary to the complex setting, there exists no unbounded Fatou component (Baker domain) for non-archimedean entire functions.

\medskip

Our paper is divided into 6 sections. The first four are aiming at the proofs of Theorems~A and~B. The last two contain applications of our main results.

Section~\ref{sec:ptwisecvg} contains a technical result that plays a key role in the sequel. 
It gives a sufficient condition for a pointwise convergent sequence of analytic functions
to have a continuous limit.
The case of  families of \emph{bounded} analytic functions is analyzed in detail in Section~\ref{sec:bdd-fam}, and a proof of Theorem~A is given as an application of these techniques. The sequential compactness of Berkovich affinoid domains in arbitrary dimension appears in a crucial manner here. Section~\ref{sec:proofB} deals with families of analytic functions with unbounded local degree, and contains a proof  of Theorem~B. In Section~\ref{sec:example}, we describe some examples to illustrate our results.

Section~\ref{sec:normal} is devoted to our notion of \emph{normality}. We discuss local conditions for
characterizing normal families, and we relate the normality locus of a rational map to its Fatou set. 

In Section~\ref{sec:entire}, we define the Fatou/Julia set of any transcendental entire map of $\aoneank$, and give its first properties.

\medskip
\noindent{ \bf Acknowledgements}: this project has been conducted during various stays of the authors
at different institutions including the CMLS at the \'Ecole Polytechnique, the Faculdad de Matem\'aticas of the PUC, and the Banff Center. We warmly thank them for the nice working atmosphere they provided to us. We thank L. Rempe for his corrections regarding complex entire trascendental maps. We are grateful with  X. Faber for many valuable comments on a first version of this work.

%
%

\section{Pointwise convergent analytic maps}\label{sec:ptwisecvg}
  Throughout, $k$ is a complete field endowed with a (non trivial) non-archimedean norm.
Recall that a basic open set of $\poneank$ is a connected component of the complement of finitely many points. 

Our goal in this section is to prove the following
  \begin{theorem}
    \label{thr:pointwise}
    Let $X \subset \aoneank$ be a basic open set. Consider a sequence 
$\{ f_n \}$ of analytic maps $f_n : X \to \aoneank$, and suppose it converges
pointwise to a function $f:X \to \aoneank$.
Then $f$ is continuous. Moreover, if $f$ is not constant, then we have $f(X \cap \H) \subset \H$.
  \end{theorem}
 More than the result itself, it is the technique involved that will be useful in the sequel. 
 The proof relies on a thorough analysis of the local degrees of the sequence $f_n$, and splits into two parts.  When the local degrees are uniformly bounded, then $f_n$ is uniformly Lipschitz for the hyperbolic metric in $\H$, and one infers the continuity of $f$  from this bound. Otherwise, one proves that the local degree explodes at one point in $X$, and the limit is constant.
 
 We emphasize that the assumption on the limit function $f$ to be valued in $\aoneank$ is crucial.
 In fact the sequence $f_n (z) = z^{p^n}$ in characteristic $p$ converges pointwise on $X = \aoneank$ to a function with values in $\poneank$ that is not continuous. We refer to Section~\ref{sec:proofs} for more examples.
 

\subsection{Basics on the Berkovich projective line}\label{sec:basics}
We refer to~\cite{baker-rumely:book} for a thorough description of this space, or to~\cite{berkovich:book}.
For sake of simplicity, we assume that $k$ is algebraically closed.

The Berkovich affine line $\aoneank$ is the set of multiplicative semi-norms on the ring $k[T]$
whose restriction to $k$ coincides with its non-archimedean norm. We shall denote by $|P(x)|\in \R_+$ the semi-norm of a polynomial $P \in k[T]$ with respect to a  point
$x\in \aoneank$.
Given $x\in \aoneank$ the set $\{ P, |P(x)| =0\}$ is a prime ideal, and $x$ induces a norm on the fraction field of $k[T]/\{ P, |P(x)| =0 \}$. One denotes by $\cH(x)$ the completion of this field with respect to $x$. It is a non-archimedean valued extension of $(k,|\cdot|)$. 

Since $k$ is algebraically closed, $x$ is determined by its values on linear functions $T- z$ with $z \in k$.
In particular, when $\{ P, |P(x)| =0\}$ is non-trivial, then $x$ is called a rigid point (or a type I point), and can be identified with  a point in $k$. The set of type I points is denoted by $\mathbb{A}^1(k)$, and for any subset $X\subset\aoneank$, we shall write $X(k)$ for the intersection $X \cap \mathbb{A}^1(k)$.

Otherwise $x$ is a norm, and falls into one of the following three categories. If its value group is equal to $|k^*|$, and the transcendence degree of  $\cH(x)$ over $k$ is equal to $1$, then $x$ is said to be of type II.
If the value group of $x$ is not equal to $|k^*|$, then $x$ is said to be of type III. Finally in the remaining case, $x$ is said to be of type IV.

One can show that for any point $x$ which is not of type IV, there exists a unique closed ball $B$  (i.e. of the form $B_r(y) = \{ z, |z- y| \le r\}$, for some $y \in k$, and $r\ge0$) such that 
$|P(x)| = \sup_B |P|$. The quantity $\sup\{  |z-z'|, z, z' \in B\}$ is called the diameter of $B$. If $x$ is not of type IV, we write $\diam (x)\in \R_+$ for the diameter of its associated ball. 

As usual, we denote the Gauss point by $x_g$ which is by definition the type II point of $\aoneank$ associated to the unit ball. 
\medskip

The set $\aoneank$ is endowed with the topology of the pointwise convergence, for which it is locally compact. There is also a natural partial order relation $x \le x'$ iff $|P(x) | \le |P(x')|$ for all $P \in k[T]$.

Given $x \in \aoneank$, the set $(x,\infty):= \{ y, y > x\}$ is a subset of type II and III points that correspond to an increasing family of balls. In particular, the diameter function on this family induces a natural homeomorphism between $(x,\infty)$ and a subset of $\R_+$.
In particular, the function $\diam$ has a natural extension to $\aoneank$ which is continuous on segments  $(x,\infty)$.
It is a fact that  any two points $x_1, x_2 \in \aoneank$ admit a maximum $\max\{ x_1, x_2\}$ for the order relation, and $(x_1,\infty) \cap (x_2,\infty) = (\max\{ x_1, x_2\},\infty)$.

It follows that  $\aoneank$ admits a natural (non-metric) $\R$-tree structure, see~\cite[Chapter 3]{valuative:tree} for a formal definition. Given any two points $x_1, x_2 \in \aoneank$, we denote by $[x_1, x_2] = \{ x_1 \le x \le\max\{ x_1, x_2\}\} \cup \{ x_2 \le x \le\max\{ x_1, x_2\}\}$, and call it the segment joining $x_1$ to $x_2$.

\medskip

The set $\aoneank$ is endowed with a natural structural sheaf of analytic functions.
If $U$ is an open subset of $\aoneank$, then $\cO(U)$ is the completion with respect to the sup norm on $U \cap k$ of the space of rational functions in $k(T)$ having poles outside $U$.
Any analytic function $f \in \cO(U)$ gives rise to a continuous map $f: U \to \aoneank$.
When $P$ is a polynomial, and $f$ is a rational map having no poles on $U$, then $ |P(f(x))| := |(P\circ f) (x)|$.

\medskip

The complement of rigid points in $\aoneank$ is denoted by $\H$. It is endowed with a natural complete metric which respects the tree structure, is invariant under the action of $\mathrm{PGL}(2,k)$, and is defined by $d_\H(x(r), x(r')) = |\log r - \log r'|$
if $x(r)$ is the point associated to the ball centered at $0$ of radius $r$.
The restriction of any analytic map $f$ to a segment $I \subset \H$ is piecewise affine in the sense that one can subdivide $I$ into finitely many segments $(x,x')$ on which $d_\H( f(x), f(x') )  = m \, d_\H ( x,x')$ for some integer $m$.

\medskip

The projective Berkovich line $\poneank = \aoneank \cup \{ \infty\}$ can be defined topologically as the one point compactification of $\aoneank$. It is convenient to view $\infty$ as a function on  $k[T]$ sending any polynomial of positive degree to $\infty$, and restricting to the standard norm on $k$. By convention $\infty$ is a type I point, so that the space of type I points in $\poneank$ is naturally in bijection with $\P^1(k)$.
The projective Berkovich line is compact, and has a natural structure of $\R$-tree for which it is complete in the sense of~\cite{valuative:tree}.

The space $\poneank$ can also be defined as an analytic curve over $k$ in the sense of Berkovich by patching together two copies of the ringed space $\aoneank$ using the
map $T \mapsto T^{-1}$. We thus have a natural notion of analytic functions on any open subset 
of $\poneank$.

\medskip

Recall that a ball in $\poneank$ is either a ball in $\aoneank$ or the complement of a ball in
$\aoneank$. It is closed (resp. open) if it is defined by a non-strict (resp. strict) inequality. 
An affinoid domain in $\poneank$ is the complement of a finite union of open balls.
The boundary of an affinoid domain is a finite set. We denote by  $\cA_Y$ the convex hull
of $\partial Y$. It is a finite tree that is called the skeleton of $Y$. 
It is a fact that the image of an affinoid domain by a non constant analytic map
is an affinoid domain.

A basic open set of $\poneank$  is a connected component of the complement of finitely many points in $\poneank$.  As for affinoid domains, the boundary of a basic open set $U$ is finite, and its convex hull is a finite tree that we denote by $\cA_U$ and refer to as the skeleton of $U$.

\medskip

Since $\poneank$ is a non-metric $\R$-tree, we may define the space of directions $T_x \poneank$ at  a point $x$ as the set of equivalence classes of  segments of the form $(x,y)$
with $x\neq y$ under the relation that identifies segments whose intersection
contains $(x,y')$ for some $y'$. For a type I or type IV point, $T_x \poneank$ is reduced to a singleton, hence they are end points of $\poneank$ for its $\R$-tree structure.
When $x$ is of type III, $T_x \poneank$ has two points, so that $x$ is a regular point. Finally when $x$  is a type II point, then  $T_x \poneank$ is isomorphic to the projective line over the residue field $\tilde{k}$ of $k$, and this isomorphism is canonical once a coordinate is fixed on the affine line. Since $\# \P^1(\tilde{k}) \ge 3$, a type II point is always a branched point in $\poneank$.

\medskip

Any analytic map $f$ defined in a neighborhood of $x$
induces a map $$T_x f : T_x \poneank\to T_{f(x)}\poneank~.$$
When $x$ is of type II, then $T_x f$ is given by a rational function under the natural identification
of $T_x\poneank$  with $\P^1(\tilde{k})$.
We shall denote by $\deg_x f$ its degree. 

It is a non-trivial fact that the function $\deg_x f$ can be extended to all points such that 
for any open set $U,V$ for which the induced map $f : U \to V$ is proper, then $ V \ni x \mapsto  \sum_{y \in f^{-1}(x) \cap U} \deg_y f$ is a constant function, see~\cite{faber:ramification,favre-letelier:equi}.

\medskip

Let  us finally fix some notation that will be used constantly in the sequel.

A direction at $x$ containing a point $x'$ will be denoted by $D_x (x')$.
We identify a direction in $T_x \poneank$ with the subset of $\poneank$ formed by all
the points in that direction.
  
Given an affinoid domain $Y$, and any $x \in \interior Y$, we  let $\val_Y(x)$ be the number of directions at $x$ pointing towards 
a point in $\partial Y$. That is, the number of connected components of $\cA \setminus \{ x\}$ where $\cA$
is the convex hull of $\partial Y \cup \{ x \}$.


\subsection{Fast directions}

In this section, we analyze the local degree on segments pointing towards infinity, and
prove a technical result (Proposition~\ref{lem:3} below) that will be applied  several times in the next sections.

\begin{lemma}[Convexity]
  \label{lem:1}
Suppose $f: Y \to \aoneank$ is a non-constant analytic function on an affinoid domain $Y \subset \aoneank$.
Then for any two points $x_0,x'\in Y \cap \H$ such that 
the segment $(x_0,x')$ contains no vertex of the convex hull of $\cA_Y \cup \{ x_0, x'\}$, the function  $  s \mapsto \log \diam f (x_s)$ is a convex, piecewise linear, and not locally constant map.
\end{lemma}

\begin{proof}
We first prove that the restriction of $\logdiam \circ f$ 
to any closed segment $I \subset \H$ is piecewise linear
with respect to the hyperbolic metric $d_\H$. It is sufficient to treat the
case of rational maps. Indeed any analytic map is a uniform limit of rational maps, 
and any two analytic maps that are sufficiently close in sup norm in a neighborhood of 
$I$ are equal on $I$.
For rational maps, the result follows from~\cite[Theorem 9.35 A]{baker-rumely:book}, by
noting that if $x,x'\in \H$
correspond to balls one contained in the other, then $d_\H (x,x') = |\log \diam (x) - \log \diam (x')|$.
Observe also that  the absolute value of the slope of this function is given by the local degree on any segment  where it is linear. In particular it cannot be locally constant.

Consider a point $x \in (x_0, x')$, and let $D_0, D_1 \in T_x \aoneank$ be the two directions
determined by $x_0$ and $x'$ respectively. Since $x$ is not a vertex of $\cA_Y\cup \{ x_0, x'\}$, at least 
one of these two directions (say $D_1$) is mapped by $T_xf$ to the direction $D$ determined by $\infty$ at $f(x)$. 

Let $m_0$ (resp. $m_1$)  be the slope of $\log  \diam \circ f$ on $(x_0,x)$ (resp. $(x, x')$) in a neighborhood of $x$. 
Suppose by contradiction that $m_0 > m_1$. 
Since $D$ points to $\infty$, we have $m_1>0$. Note that $x$ is automatically a type II point.   Then $T_x f$ is a rational map of degree at least $m_0$, and
there necessarily exists at least one direction $D_2$ at $x$ different from both $D_0$ and $D_1$ which is mapped to $D$. 
By assumption $x$ cannot be a vertex of the convex hull of $\cA_Y \cup \{ x_0, x' \}$, so that 
the open ball $B$ determined by $D_2$ in $\poneank$ is actually contained in $Y$.
But then $f(B)$ would contain $\infty$ which contradicts $f(Y) \subset \aoneank$.
\end{proof}

For convenience, we introduce the following definition.
\begin{definition}
  \label{def:1}
  Let $Y \subset \aoneank$ be an affinoid domain and $f: Y \to \aoneank$ be an analytic function.
A {\sf fast direction} $D \in T_x \aoneank$ for $f$ at $x$
is a direction determined by a point in $\partial Y$ such that $\infty \in T_xf (D)$ and
maximizing $\deg_D T_x f$ over all such directions.
\end{definition}
For any non-constant function $f$ and any point  $x \in \interior Y$, the map $T_xf : T_x \poneank \to T_{f(x)}\poneank$ is surjective. In particular, 
any point  $x \in \interior Y$ admits a fast direction. 
\begin{definition}
Given any $x \in \interior Y$, a {\sf  fast arc}  at $x$ is a segment
$[x,x']$ such that for all $y \in [x,x')$, the direction at $y$
determined by $x'$ is a fast direction.
\end{definition}
Note that the restriction of $f$ to a fast arc is injective and increasing to $\infty$.
\begin{lemma}
  \label{lem:2}
 Any point $x \in \interior Y$ admits a fast direction $D$ such that 
$$\deg_D T_x f \ge \dfrac{\deg_x f}{\val_Y (x)}~.$$
\end{lemma}

\begin{proof}
If $x$ is not of type II, then $\deg_D T_x f = \deg_x f$ for any direction at $x$.
Suppose $x$ is of type II. Then 
$\sum \deg_D T_x f  = \deg_x f $ where the sum ranges over all directions $D$ that are mapped to infinity by $T_xf$, which implies the result, since all these directions $D$ must point to an element of $\partial Y$. 
\end{proof}

We are now in position to prove the following key result.

\begin{proposition}
  \label{lem:3}
  Consider an affinoid domain $Y \subset \aoneank$. Then there exists a constant $C >0$ such that, for all non-constant analytic maps $f: Y \to \aoneank$ and all  $x_0 \in \interior Y$, there exists $x' \in \partial Y$ with the following properties:
  \begin{enumerate}
  \item $[x_0,x']$ is a fast arc (for $f$);
  \item $\deg_{x'} f \ge C \, {\deg_{x_0} f}$;
  \item $\deg_{D_{x}(x')}T_x f \ge  C \cdot {\deg_{x_0} f},$ for all $x \in [x_0, x')$.
  \end{enumerate}
\end{proposition}

As an immediate consequence we obtain
\begin{corollary}
  \label{cor:1}
For any affinoid domain $Y \subset \aoneank$, there exists  a constant $C'>0$ such that for any analytic map $f : Y \to \aoneank$,  we have
$$
\sup_{x\in Y} \deg_x f \le C' \sup_{x \in \partial Y} \deg_x f~.$$
\end{corollary}

\begin{proof}[Proof of Proposition~\ref{lem:3}]
If a direction $D$ at a point $x \in \interior Y$ satisfies $\infty \in T_xf (D)$ then $D$ is necessarily determined by a point in $\partial Y$.  It is thus always possible to find a fast arc $[x_0, x']$ joining $x_0$ to a point $x' \in \partial Y$. 
Write $[x_0, x'] = [v_0, v_1] \cup  [v_1, v_2] \cup ... \cup  [v_\ell , x']$
such that $x_0=v_0$, and $v_i$ are branched points of the convex hull of $\cA_Y \cup \{ x_0 \}$, and the interior of each segment does not contain any other branched point of this finite tree.

Denote by $x_s$ the unique point of $[x_0, x']$ at hyperbolic length $s$ of $x_0$, and
by $D_s$ the direction at $x_s$ pointing towards $x'$. 
Recall that the slope $\Delta(s)$ of $s \mapsto \log \diam f (x_s)$ is equal to $\deg_{D_s} f$
which is in turn bounded from above by $\deg_{x_s} f$.

By Lemma~\ref{lem:2}, we have $\Delta(0) \ge \deg_{x_0} f/\val_Y(x_0)$, and by Lemma~\ref{lem:1}, $\Delta (s) \ge 
 \deg_{x_0} f /\val_Y(x_0)$ as long as $x_s \in [x_0, v_1)$. At $x_{s_1}:= v_1$, we apply again Lemma~\ref{lem:2} so that  
 $$\Delta(s_1) \ge \frac{\deg_{v_1} f}{\val_Y(v_1)} \ge \frac{\deg_{D_{v_1}(x_0)} f}{\val_Y(v_1)}
 \ge \frac{\deg_{x_0} f}{\val_Y(x_0)\, \val_Y(v_1)}~.$$
Iterating the argument we finally end up with the bound
$\Delta(s) \ge C \deg_{x_0}f$ for all $s$ with $C$ being the inverse of the product from $i=0$ to $\ell$ of $\val_Y(v_i)$. Note that this bound is uniform in $x$ since $\ell$ is bounded from above by the $1$ + the number of branched points of the skeleton of $Y$; and $\val_Y(x)$ is always less than $\# \partial Y$.

We conclude the proof noting that $\deg_{x'} f \ge \deg_{D_s} f$ 
for any $x_s \in [x_0, x']$ sufficiently close to $x'$.
\end{proof}


\subsection{The case of unbounded degree}
In this section we prove Theorem~\ref{thr:pointwise} in the case the local degree 
explodes. 
\begin{proposition}\label{prop:3}
Let $X \subset \aoneank$ be a basic open set. Consider a sequence of analytic maps  $f_n : X \to \aoneank$  such that 
$\limsup_n\diam f_n(x) < \infty$, for all $x \in X \cap \H$.
If $f_n(x_0)$ converges in $\aoneank$, and  $\deg_{x_0} f_n \to \infty$, for some $x_0 \in X$, then
$f_n$ converges pointwise to a constant. Moreover, for all $x \in X$, we have that $\diam f_n(x) \to 0$.
\end{proposition}
Note that the limit function might be a constant in the hyperbolic space $\H$.
\begin{proof}
By Corollary~\ref{cor:1}, we may (and shall) assume $x_0 \in \H$, and write $z_0 = \lim f_n(x_0)$.
We begin with the following
\begin{lemma}\label{lem:sorry-jan}
Let $X \subset \aoneank$ be a basic open set.
Suppose $f_n : X \to \aoneank$ is a sequence of analytic functions such that 
$\limsup_n\diam f_n(x) < \infty$, for all $x \in X \cap \H$.  If $x_0 \in \H$ is a point such that $\deg_{x_0} f_n \to\infty$, then we have $\diam f_n(x_0) \to 0$. 
\end{lemma}

Now pick any $x \in X \cap \H$, and  consider an affinoid domain $Y \subset X$ containing both $x_0$ and $x$. Let $\cA$ be the convex hull of $\partial Y \cup \{ x, x_0\}$. Denote by $y_0 := x_0, y_1,\dots ,  y_{N+1} := x$ the consecutive vertices of $\cA$ lying on $[x_0, x]$ so that 
$[x_0, x] = \cup_{i=0}^N [y_i, y_{i+1}]$. We shall prove by induction that $\diam f_n(y_i) \to 0$ and $\lim f_n(y_i) = z_0$. There is nothing to prove for $i=0$. Suppose $\diam f_n(y_{i-1}) \to 0$.

By Lemma~\ref{lem:1} the function  $\logdiam \circ  f_n$ is convex and non constant on $[y_{i-1}, y_i]$. 
Hence there exists a unique $c_n \in  [y_{i-1}, y_i]$ such that 
$\diam f_n (c_n) \leq \diam f_n (y)$ for all $y \in
 [y_{i-1}, y_i]$. In particular $ \diam f_n (c_n) \leq \diam f_n (y_{i-1}) \to 0$ by the induction hypothesis. Since $\logdiam $ is increasing on the segment $[c_n, y_i]$, and convex with slope equal to the local degree,  we conclude that $\logdiam f_n (y_i)  \le \logdiam f_n(c_n) +  d_\H(c_n, y_i) \cdot \deg_{y_i} f_n$. 
By contradiction, assume that there exists $\varepsilon > 0$ and a subsequence
$f_{n}$ such that $\diam f_{n}(y_i) > \varepsilon$.
If, passing to a further subsequence, $\deg_{y_i} f_n$ is bounded, then we get $\diam f_n (y_i) \to 0$. Otherwise,  $\deg_{y_i} f_n \to \infty$ and  
Lemma~\ref{lem:sorry-jan} also shows 
that  $\diam f_n (y_i) \to 0$, which is impossible. Hence $\diam f_n (y_i) \to 0$.

To show that $f_n(y_i) \to z_0$, we pick a basic open set $V$ containing 
$z_0$.  We must show that for $n$ sufficiently large $f_n(y_i) \in V$.
There exists $r >0$ such that, if $y \in V$ and $\diam y \le r$,
then the closed ball $B_r(y) \subset \aoneank$ with diameter $r$
containing $y$ is contained in $V$.
Hence, if $f_n(y_{i-1}) \in V$ and $\diam f_n (y_{i-1}) \le r$, then $f_n(c_n) \in  V$.
Since the diameter is increasing from $f_n(c_n)$ to $f_n(y_i)$, and the latter is at most $r$,
for $n$ sufficiently large, we also have that $f_n(y_i) \in V$.
Therefore, $f_n(y_i) \to z_0$.  

Now for any $x \in X (k)$, take a small ball $B$ containing $x$ with boundary point $x_B \in X$. 
Thus, $f_n(x_B) \to z_0$ and it follows that $f_n(x)$ also converges to $z_0$.
\end{proof}

\begin{proof}[Proof of Lemma~\ref{lem:sorry-jan}]
Consider an  affinoid domain $Y \subset X$ which contains $x_0$ in its interior.
Suppose by contradiction that $\limsup \diam f_n(x_0) >0$.
Passing to a subsequence we may assume that $\diam f_n (x_0) \ge \varepsilon >0$,
and for some $x' \in \partial Y$, 
the arc $[x_0,x']$ satisfies the conditions of Proposition~\ref{lem:3} for all  $f_n$'s. 
Since $[x_0, x']$ is a fast arc, the function $f$ is injective and increasing to infinity.
From the third condition  of Proposition~\ref{lem:3}, we conclude that $\deg_{x} f_n \ge C \, \deg_{x_0} f_n$ for all  $x \in [x_0,x']$.
Therefore,
$\log \diam f_n(x') \ge \log \diam f_n (x_0) + C \, d_\H(x',x_0) \, \deg_{x_0} f_n  \longrightarrow \infty$,
which contradicts the hypothesis of the Lemma.
\end{proof}


\subsection{Proof of Theorem~\ref{thr:pointwise}}
We may always assume $f$ to be non constant. Pick $\bar{x}\in X$, and $Y$ an affinoid neighborhood of $\bar{x}$.  We shall prove that $f|_Y$ is continuous and $f(Y \cap \H) \subset \H$.

If $\sup_Y \deg_x f_n \to \infty$, then $\deg_x f_n \to \infty$ for some $x \in \partial Y$ by Corollary~\ref{cor:1}. This is excluded by Proposition~\ref{prop:3} and our standing assumption. Whence 
\begin{equation}\label{eq:bddeg}
\sup_{x\in Y} \deg_x f_n  \le D
\end{equation} for some $D\ge0$. This bound is not sufficient for our purposes. But one can prove:
\begin{equation}\label{eq:bd}
\sup_{x\in \aoneank} \sup_{n \ge0} \# f_n^{-1} (x) \cap Y \le D\cdot \# \partial Y~.
\end{equation}
Indeed, if $f_n : Y \to f_n(Y)$ is a proper map then $\displaystyle{d(y) : = \sum_{x\in f_n^{-1}(y)\cap Y} \deg_x f_n}$ is constant
on $f(Y)$. Noting that $f_n^{-1} (f_n (\partial Y)) \subset \partial Y$,  and computing $d(y)$  at a point in $f_n(\partial Y)$, we get
$\# f_n^{-1} (x) \cap Y \le d(y) \le D\cdot \# \,\partial Y$. When $f_n$ is not proper onto its image, we can only conclude that  the maximum of $y \mapsto d(y)$ is attained at a point in $f_n(\partial Y)$. But this is sufficient to get the bound in~\eqref{eq:bd}.
\begin{lemma}
  \label{pro:2}
The function $f$ is sequentially continuous.
\end{lemma}
We give a proof of this fact thereafter. Let us prove that $f$ is then continuous at $\bar{x}$.
We proceed by contradiction. That is, there 
exists a basic open set $V$ containing $f(\bar{x})$ such that for all open sets 
$U \subset Y$ that contain $\bar{x}$ there exists $x_U \in U$ such that $f(x_U) \notin V$. 
  It is sufficient to  construct a sequence $\{ x_m \}_{m\ge1} \subset X$ converging to $\bar{x}$ such that $f(x_m) \notin V$ for all $m \geq 1$. We proceed inductively. 
Let $U_1=\mathrm{int}\, Y$ and  $x_1 = x_{U_1}$. 
We may assume that $x_m$ and $U_m$ have already been constructed.
If $x_m$ is not in the direction
of  infinity at $\bar{x}$, then let $B_m$ be a closed ball  containing $x_m$ and not containing $\bar{x}$ such that $\diam \bar{x} - \diam B_m <1/m$.
Otherwise let $D_m$ be an open ball containing $\bar{x}$ and not containing $x_m$ such that $\diam D_m - \diam \bar{x} < 1/m$. In the former case we let $U_{m+1} = U_m \setminus B_m$ and in the latter $U_{m+1} = U_m \cap D_m$. Now let $x_{m+1} = x_{U_{m+1}}$. 
Passing to a subsequence, the directions $\{ D_{\bar{x}}(x_m) \}$ 
are either pairwise distinct or coincide for all $m$. In both cases,  we get
$x_m \to \bar{x}$ by construction.  
This shows $f$ is continuous.

Now since $f$ is continuous and not constant on $Y$, one can find a point $x \in Y \cap \H$ whose image $f(x)$ belongs to $\H$. Pick any other $x' \in Y\cap \H$. Then 
$$
d_\H (f(x), f(x'))
\le \liminf_n 
d_\H (f_n(x), f_n(x'))
\le D \, d_\H ( x, x')<\infty~.$$
We have thus proved $f(Y\cap \H) \subset \H$. This concludes the proof of Theorem~\ref{thr:pointwise}.

\begin{proof}[Proof of Lemma~\ref{pro:2}]
We proceed by contradiction. Assume that $f$ is sequentially discontinuous at $\bar{x}$.
Then there exists a basic open set $V$ containing $f(\bar{x})$ and a sequence $x_m \to \bar{x}$ 
such that $f(x_m) \notin V$ for all $m$.

Without loss of generality,  for all $m$ we have that $x_m \in Y$ and $f(x_m) \in B$ where $B$ is a connected component of $\aoneank \setminus V$,  and $f_n(\bar{x}) \in V$ for all $n$.
Denote by $z_B$ the boundary point of $B$.
Let $y_{m}(n)$ be the closest preimage of $z_B$ to $\bar{x}$ under the map $f_n:[x_m,\bar{x}] \to
\aoneank$. Then, $d_\H (\bar{x}, y_m(n)) \ge d_\H (f_n(\bar{x}), z_B)/ D$. Observe that $\liminf  d_\H (f_n(\bar{x}), z_B) > 0$ since  closed hyperbolic balls are closed in the topology of $\aoneank$.
Hence, there exists $\varepsilon > 0$ such that $d_\H (\bar{x}, y_m(n)) > \varepsilon$. Fix $n=n_0$, and observe that $\{ y_m(n_0)\}_{m \in \N}$ has 
at most $D$ elements. 
Passing to a subsequence of $x_m$, we may assume  $y_m(n_0) =y$ is a fixed point in $\H$.
But $x_m \to \bar{x}$ in $\aoneank$, and $y \in (x_m,\bar{x})$ so that $d_\H(y,\bar{x}) =0$, a contradiction. 
\end{proof}

\begin{remark}
The argument above to obtain continuity from sequential continuity can be generalized as follows.
Let $X$ be any affinoid domain (of arbitrary dimension), $S$ any topological space, and let $f : X \to S$ be any sequentially continuous map. Then $f$ is continuous.

\medskip
Indeed pick any $\bar{x} \in X$, and any open subset $V \ni f(\bar{x})$. Suppose by contradiction that  $\bar{x}$ lies in the closure of $\mathcal{B} = \{ x \neq \bar{x}, \, f(x) \notin V \}$. Since $X$ is an angelic space (see~\cite{poineau:sequential}), one can find a sequence  $x_n \in \mathcal{B}$ such that $x_n \to \bar{x}$, which is impossible by assumption.
\end{remark}

%
%

\section{Family of analytic functions with bounded local degree}\label{sec:bdd-fam}

In this section, we give a proof of our first main result Theorem~A.
To do so, we first deal with the case of analytic maps with values in an affinoid domain, and we prove the following key result.
\begin{theorem}\label{thm:extract-affinoid}
Let $U$ be a basic open set and $Y \subsetneq \poneank$ be an affinoid domain.
Then any sequence of analytic functions $f_n: U \to Y$ 
admits a subsequence that is pointwise converging on $U$ to a continuous function.
\end{theorem}

The above result is false if $U$ is assumed to be an affinoid domain.
In fact,  let $B \subset \aoneank$ be the closed unit ball containing $z=0$
and consider $f_n: B \to B$ defined by $f_n(z) = z^n$. It follows that
any pointwise convergent subsequence has a limit $f$ fixing the Gauss point $x_g$ and $f$ is constant equal to $0$ in the open unit ball containing $z=0$ (compare with Section~\ref{sec:NotC0}).

\subsection{Family of rational functions}

Let us first analyze the special case  of families of rational maps
with fixed poles and uniformly bounded degree.

\begin{proposition}\label{prop:rational}
Suppose $k$ is algebraically closed and choose $z_1,\dots,z_p \in k$.
Consider a sequence, indexed by $n \ge 0$, 
$$(a^{(n)}_l, a^{(n)}_{i,j}) \in k^{d+1}\times k^{pd},$$  
where $0 \le l \le d$, $ 1 \le i \le d$,  $1 \le j \le p$.

If  $(a^{(n)}_l, a^{(n)}_{i,j})$ converges to a point in $\mathbb{A}^{d+1+pd,an}_k$, then
 the function $$P_n(z) = 
\sum_{l=0}^d a^{(n)}_l z^l  +
\sum_{j=1}^p \sum_{i=1}^d  \frac{a^{(n)}_{i,j}}{(z-z_j)^i}$$ converges pointwise to a continuous function $P: \aoneank \setminus\{ z_1,\dots, z_p \} \to \aoneank$.
\end{proposition}
Recall that $\mathbb{A}^{d+1+pd,an}_k$ denotes the set of multiplicative semi-norms on the 
ring of polynomials in $d+1+pd$ variables with coefficients in $k$ that restricts to the given norm on $k$. 
\begin{proof}
We only need to show that $ P_n(z)$ converges in $\aoneank$ for any $z \in \aoneank\setminus\{ z_1, ... , z_p \} $.
Then  the fact that $P = \lim_n P_n$ is a continuous function will follow from Theorem~\ref{thr:pointwise}.

Write $\alpha = \lim_n (a^{(n)}_l, a^{(n)}_{i,j})$.
We consider first the case of a rigid point $z \in k$. For any $w \in k$, define the following polynomial in $d+1+pd$ variables:
$\phi_w(T_l, T_{i,j}) = 
\sum_{l=0}^d T_l z^l  +
\sum_{j=1}^p \sum_{i=1}^d  \frac{T_{i,j}}{(z-z_j)^i} -w$. Then 
$$
|P_n(z) -w| =
|\phi_w ( a^{(n)}_l,a^{(n)}_{i,j})|  \to |\phi_w( \alpha)| ~.
$$
Since $k$ is algebraically closed, $ P_n(z)$ converges in $k$.

Next we consider the case of a type II point $z$.
Recall that, given $\zeta$ such that $|z - \zeta| \leq \diam(z)$,
for all $w \in k$:
\begin{equation}\label{eq:convenient}
|z-w| = \max \{ \diam(z), |\zeta -w| \}. 
\end{equation}

In order to prove that $P_n(z)$ converges, we need to show that
$|P_n(z) - w|$ converges for all $w \in k$. In fact,
if $\zeta' \in k$ is such that $|z - \zeta'| \leq \diam(z)$ and $|\zeta'-z_j| \ge \diam(z)$ for all $j$, then $|P_n(\zeta') - P_n(z)| \le \diam P_n(z)$.
Since $|P_n(\zeta') -w|$ is convergent, taking $\zeta = P_n(\zeta')$ in~\eqref{eq:convenient}, we just need to prove that $\diam P_n(z)$ converges.
To show that $\diam P_n(z)$ converges, choose $\zeta_0,\dots , \zeta_{p(d+1)+1} \in k$
such that
\begin{itemize}
\item
$|\zeta_i - z| = \diam(z)$ for all $i$;
\item
$|\zeta_i - \zeta_j| = \diam(z)$ for all $i\neq j$;
 \item
$|\zeta_i - z_j| \ge \diam(z)$ for all $i,j$.
\end{itemize}
Since the degree of $P_n$ is bounded by $p(d+1)$, at least two different directions
determined by the $\zeta_i$'s at $z$ are mapped to distinct directions at $P_n(z)$.
Hence the three conditions above ensure 
$$
\diam P_n(z) = \max \{ |P_n(\zeta_i) - P_n(\zeta_j)|, \, i \neq j\}~,
$$
which is convergent by the same argument as above.

Finally if $z \in \aoneank$ is a point of type III or IV, we pick a segment $I \subset \H$
of positive and bounded length, containing $z$.
Denote by $z_s$ the unique point in $I$ at hyperbolic distance $s$ from $z$.
By Lemma~\ref{lem:sorry-jan}, the function $\delta_n(s) := \log \diam P_n(z_s)$ is  Lipschitz for the hyperbolic metric with Lipschitz constant $\le d$, hence forms a family of equicontinuous functions. Since type II points are dense on any non trivial segment of $\H$, we know that $\delta_n(s)$ is converging pointwise on a dense set of $I$. Whence $\delta_n(s)$ converges on $I$ to a continuous function (possibly $\equiv - \infty$).

This concludes the proof.
\end{proof}


\subsection{Proof of Theorem~\ref{thm:extract-affinoid}}
Since the range of all maps is included in an affinoid domain, any pointwise limit is necessarily continuous by Theorem~\ref{thr:pointwise}.
We only have to prove the existence of a subsequence that converges pointwise.

Let us first prove the theorem under the assumption that $k$ is algebraically closed.

Without loss of generality we may assume $Y$ is the unit ball.
If  $\deg_{x_0} f_n$ is unbounded for some $x_0 \in U$, then after passing to a subsequence
$f_n(x_0)$ is converging in $Y$, and we may apply Proposition~\ref{prop:3}.
We conclude that $f_n$ converges pointwise to $\lim f_n(x_0)$.

From now on, we suppose that $\deg_x f_n$ is bounded for all $x \in U$.
Consider an affinoid domain $X \subset U$. Since $U$ is the countable union of 
affinoid domains, a diagonal argument shows that it is sufficient to establish the pointwise convergence of an appropriate subsequence in $X$.

By Proposition~\ref{lem:3}, there exists $D \geq 0$ such that $\deg_x f_n \le D$  for all $x \in X$ and for all $n$.
Extracting a subsequence,
we may assume that either there exists $x_0 \in X \cap \H$ such that 
$\lim_n \diam f_n(x_0) =0$ or, 
$\inf_n \diam (f_n(x)) >0$ for all $x \in X \cap \H$.

In the first case, since the local degree are uniformly bounded on $X$, we get 
$f_n(x) \to f(x_0)$ for all $x \in X \cap \H$.  
Since balls of sufficiently small diameter are mapped onto balls, $f_n (x) \to f(x_0)$ for all $x \in X$.

In the second case, 
we use the Mittag-Leffler decomposition on affinoid domains of the affine line, see~\cite[p.7]{fresnel}.
For any $d$ write $f_n = P^d_n + R^d_n$ with $P^d_n$ a rational map
of degree $\le d$ and $\sup_{X(k)} |R^d_n| \le \eta_d$ for  a sequence $\eta_d \to 0$ as $d \to\infty$. 
Since affinoid domains are sequentially compact by~\cite{poineau:sequential}, using a diagonal extraction argument we may assume that for each $d$ the coefficients of $P^d_n$ converge in the Berkovich affine space of the suitable dimension.  Proposition~\ref{prop:rational} 
then shows that $P^d_n \to P^d$ pointwise on $X$ for all $d$.

If $x \in X \cap \H$, pick $d$ large enough such that $\diam (f_n(x)) > \eta_d$ for all $n$.
Then $f_n(x) = P^d_n(x)$ for all $n$ so that $f_n(x) \to P^d(x)$. 
If $x \in X(k)$ is a rigid point, and $\diam P^d(x) > \eta_d$ for some $d$ then $f_n(x) = P^d_n(x) + R^d_n(x) \to P^d(x)$. Otherwise $\diam P^d(x) \le \eta_d$ for all $d$. Pick $\e \ll 1$, and $d \gg1$ such that $\eta_d \le \e/2$. As $P^d_n(x)$ converges to a point of diameter $\le \eta_d$, there exists an $N$ such that for $n,m \ge N$, we get $|P^d_n(x) -P^d_m(x)| \le 2 \eta_d$. Whence
$|f_n(x) - f_m(x) | \le \max\{\eta_d, 2 \eta_d\}  \le \e$. This proves $f_n(x)$ is a Cauchy sequence and converges by completeness.

\medskip
Now we establish the theorem when  $k$ is not algebraically closed. We start by 
embedding $ k$ into $\bar{k}$, where $\bar{k}$ denotes the completion of an algebraic closure of $k$.
Denote by $U_{\bar{k}}$ and $Y_{\bar{k}}$ the subsets of $\P^{1,an}_{\bar{k}}$ obtained as 
a lift of $U$ and $Y$, respectively, under the quotient $ \P^{1,an}_{\bar{k}} \to \poneank$ by the 
Galois action.
It follows that $f_n$ lifts to a map $\bar{f}_n: U_{\bar{k}} \to Y_{\bar{k}}$ such that 
$\bar{f}_n \circ \sigma = \sigma \circ \bar{f}_n$ for all $\sigma \in \gal (\bar{k}/k)$.
Thus we may extract a subsequence $\bar{f}_{n_j} \to \bar{f}$
which is pointwise convergent in $U_{\bar{k}}$. 
Since Galois group elements act continuously and $\bar{f}$ is continuous,
$\bar{f} \circ \sigma = \sigma \circ \bar{f}$ 
for all $\sigma \in \gal (\bar{k}/k)$. 
Recall that $U \subset \aoneank$,  and $U =U_k$ is isomorphic to 
$U_{\bar{k}}$ modulo the action of  $\gal (\bar{k}/k)$. Similarly, 
$Y$ is isomorphic to $Y_{\bar{k}}$ modulo the Galois action.
Whence $\bar{f}$ descends to a continuous map $f:U \to Y$ which is the pointwise limit of $f_{n_j}$.
\hfill $\Box$

%

\subsection{Proof of Theorem~A}

We shall rely on the following two results that are consequences of Theorem~\ref{thm:extract-affinoid}. Recall that $x_g$ denotes the Gauss point.

\begin{lemma}
  \label{lem:6}
  Let $X$ be the affinoid obtained after removing finitely many directions at $x_g$ from 
$\poneank$. Consider a sequence $f_n : X \to \poneank$  of analytic maps
such that $f_n^{-1} \{x_g\} = \{ x_g \}$, and $\deg_{x_g} f_n \to \infty$.

Then there exists a subsequence $f_{n_j}$ converging pointwise in $X$.
\end{lemma}

\begin{lemma}
  \label{lem:7}
  Let $U$ be an open annulus and $f_n : U \to \poneank \setminus \{ 0, \infty \}$ a sequence of analytic maps. Then there exists a subsequence $f_{n_j}$
converging pointwise in $U$.
\end{lemma}

Since any connected open sets of $\poneank$ is a countable union of basic open sets, we may assume $X$ is a basic open set. Decompose $X$ into a finite union of
open annuli and affinoids as follows.
Let $\pi_X: X \to \cA_X$ be the natural  retraction of $X$ on its skeleton. For each open edge $I=(v_1,v_2)$ of $\cA_X$, the set  $A_I= \pi_X^{-1}(I)$ is an open annulus and
for each branched point $v$ of $\cA_X$ the affinoid $X_v = \pi_X^{-1}(v)$ is the complement of finitely many directions at $v$.

Apply Lemma~\ref{lem:7} finitely many times, to extract a subsequence $f_n$ converging pointwise in the union of the annuli $A_I$. 

Now let $Y$ be the affinoid $X_v$ associated to a vertex  $v$ of $\cA_X$.
 Passing to a subsequence, let $w = \lim f_n(v)$. If necessary,
passing to a further subsequence, $f_n(v) = w$ for all $n$ or
$f_n(v) \neq w$ for all $n$. 

Observe that $f_n(X_v)\cap (0,\infty)$ is reduced to $f_n(v)$ (if non empty). It follows that 
in the latter case $f_n$ converges to $w$ in $Y$.

If $f_n(v) =w$ for all $n$, and $\deg_v f_n \to \infty$, then we apply  Lemma~\ref{lem:6} to extract a pointwise convergent subsequence in $Y$.
We may thus assume that $f_n(v) =w$, and $\deg_v f_n = d$ for all $n$.
Let $I_1, \dots, I_\ell$ be all the edges of $\cA_X$ with one endpoint at $v$ and consider the basic open set  
$U = Y \cup A_{I_1} \cup \cdots \cup  A_{I_\ell}$. 
If $w \notin (0,\infty)$, then $w$ is in a direction $D$ of a point $x \in (0,\infty)$. It follows $f_n(U) \subset D$ for all $n$.
Thus we may apply Theorem~\ref{thm:extract-affinoid} to extract a subsequence that is pointwise converging in $Y$.
If $w \in  (0,\infty)$, then the degree $d_j$ along $I_j$ is constant for all $j$ such that $f_n(I_j) \subset (w,\infty)$, by Lemma~\ref{lem:1} applied to $f_n$ and $1/f_n$. Therefore $d_j$ is bounded by $d$. After observing that if $x \in U$ and $f_n(x) \in (w,\infty)$, then $x \in I_j$ for some $1 \le j \le \ell$ we have that a neighborhood of $f_n(U) \subset B$ for some closed ball $B$ and all $n$.
Applying Theorem~\ref{thm:extract-affinoid} we obtain the desired subsequence converging pointwise in $Y$ and Theorem~A follows.

\begin{proof}[Proof of Lemma~\ref{lem:6}]
Recall that  we identify a direction $D$ at $x_g$ with the open ball in $\poneank$ of points
determining $D$. As $f_n^{-1}\{ x_g \} = \{ x_g\}$ for all $n$,  then for any direction $D$ determined by points in $X$,   $f_n(D)$ is the ball determined  by the direction 
$T_{x_g} f_n (D)$.

 Since $\deg_{x_g} f_n \to \infty$, without loss of generality we may assume that $T_{x_g} f_n \neq T_{x_g} f_m$ provided that $n \neq m$. 
  Let $S \subset \P^1(\tilde{k})$ be the  set of directions $D$ at $x_g$ determined by points in $X$, and  such that $f_n(D) = f_m(D)$ for some $n, m$ with $n \neq m$. Observe that $S$ is countable.
 
In the directions not in $S$, the sequence $f_n$ converges pointwise to  $x_g$.   
From Theorem~\ref{thm:extract-affinoid} and a  diagonal argument we may extract a subsequence $f_{n_j}$ converging pointwise in $D$ for all directions $D$ in $S$. 
 \end{proof}

\begin{proof}[Proof of Lemma~\ref{lem:7}]
It is sufficient to prove that we may extract a pointwise convergent subsequence
in an (open) annulus $Y \subset U$ such that $\bar{Y} \subset U$.
Write $[a_Y, b_Y] =\cA_Y \subset \cA_U$.

Relabelling $a_Y, b_Y$, if necessary, and passing to a subsequence we may assume that there exists a ball $B$ (containing $0$ or $\infty$) such that either $f_n (Y) \subset \poneank \setminus B$, or $f_n(a_Y) \to 0$ and $f_n(b_Y) \to \infty$. 
In the first case we obtain a pointwise convergent subsequence from Theorem~\ref{thm:extract-affinoid}. 

In the latter case, we first observe that $f_n^{-1} (0,\infty) \subset (a_Y, b_Y)$ and $f_n(a_Y,b_Y) \subset (0,\infty)$ for all $n$ sufficiently large, since $f(Y) \subset \poneank \setminus \{ 0,\infty\}$.
By Lemma~\ref{lem:1} applied to $f_n$ and $1/f_n$,  
there exists $d_n \ge 1$ such that $\deg_x f_n = d_n$ for all $x \in \cA_Y$.

First we prove pointwise convergence in $\cA_Y$.
Denote by $x_s \in \cA_Y$ the point in $(a_Y,b_Y)$ at hyperbolic distance $s$ from $a_Y$, so that
$$\log \circ \diam (f_n(x_s)) = d_n (s - s_n)$$
for some $s_n \in \R$, after passing to a subsequence and maybe
 postcomposing by $1/z$.
Passing to further subsequences, let $t=\lim_n s_n \in [-\infty, +\infty]$ and
$\beta = \lim_n d_n(t-t_n) \in [-\infty, +\infty]$.
If $s \neq t$,  then $f_n(x_s) \to \pm \infty$. Now $f_n(x_t)$ converges to the unique point $y_\beta$ in $[0,+\infty]$ of diameter $\exp(\beta)$ with the obvious interpretation when $\beta = \pm \infty$. 

For every $x_s \in (a_Y,b_Y)$, let $X_s$ be the union of the directions at
$x_s$ not determined by $a_Y$ or $b_Y$. It follows that for $s \neq t$, in $X_s$, the maps $f_n$ converge
pointwise to $\lim_n f_n(x_s) = 0$ or $\infty$. Similarly, if $\beta = \pm \infty$, then  $f_n$ converges pointwise in $X_t$ 
to $\lim_n f_n(x_t) = 0$ or $\infty$.
So we assume that $\beta \in \R$ and we have to prove that a subsequence converges pointwise in $X_t$. In fact, passing to a subsequence, either $f_n(x_t) = y_\beta$ for all $n$ or $f_n(x_t) \neq y_\beta$ for all $n$. In the latter case we have pointwise convergence to $y_\beta$ in $X_t$. In the former, we may apply the previous Lemma~\ref{lem:6}
 to conclude that $f_n$ has a pointwise converging subsequence in $X_t$.
\end{proof}

%
%

\section{Montel's theorem}
\label{sec:proofB}

In this section we give a proof of Theorem~B.
To do so, we first need to analyze in more detail families of analytic functions with unbounded local degree that avoid three points in $\poneank$.

\subsection{Family of analytic functions with unbounded local degree}\label{sec:unbdd-fam}

Recall that $\cA_Y$ denotes the skeleton, and $\partial Y$ the boundary of a basic set (or of an affinoid domain). 

Our aim is to prove the following two results.

\begin{proposition}\label{prop:subfine}
Let $X$ be any basic open subset of $\poneank$, and $f_n : X \to \poneank \setminus \{ 0,1 , \infty\}$ be any sequence of analytic functions such that  $\deg_x (f_n) \to \infty$ for some $x\in X$. Then replacing $f_n$ by a suitable subsequence we are in one of the following two situations:
\begin{itemize}
\item either $f_n$ is converging pointwise  on $X$ to a constant function;
\item
or there exists a branched point $x'$ of $\cA_X$ such that $f_n(x') = x_g$ for all $n$ and $\deg_{x'} f_n \to \infty$.
\end{itemize}
\end{proposition}

\begin{proposition}\label{prop:subfine2}
Let $X$ be a basic open subset of $\poneank$, $x_0 \in X$, and $f_n : X \to \poneank \setminus \{ 0,1 , \infty\}$ be a sequence of analytic functions such that  
$f_n^{-1}(x_g) = x_0$ for all $n$, and  $\deg_{x_0} (f_n) \to \infty$.
Then $\unsepdeg_{x_0} (f_n) \to \infty$ (so that the residual characteristic of $k$ is positive).
\end{proposition}

\begin{proof}[Proof of Proposition~\ref{prop:subfine}]
Suppose first  that the set of integers $n$ such that $x_g \notin f_n(X)$ is infinite. 
Then we can find a subsequence such that $f_{n_j}(X)$ is included in a closed ball of $\poneank$  having $x_g$ as a boundary point. By relabeling $0,1,\infty$,  and possibly extracting again a subsequence, we may suppose $f_{n_j}(X) \subset \{ |z| \le 1\}$ for all $j$.
By Proposition~\ref{prop:3}, we conclude that some subsequence is converging
to a constant function as required.

From now on, we assume $f_n^{-1}(x_g) \cap X$ is non empty for any $n$.
Pick any point  $x' \in f_n^{-1}(x_g) \cap X$.
Since $f_n(X)$ avoids the triple $\{ 0,1,\infty\}$, the point $x'$ is 
necessarily a branched point of $\cA_X$. We may thus
assume that $f_n^{-1}(x_g)$ is independent of $n$.
If $x \in f_n^{-1}(x_g)$, then the proposition follows. Hence we also assume that
$x \notin f_n^{-1}(x_g)$.

Let $U$ be the connected component of $X \setminus f_n^{-1}(x_g)$ containing $x$.
Passing to a subsequence, $f_n(x) \to y$ and, maybe after postcomposition by a fixed projective transformation, $f_n(U)$ is contained in the unit ball.
By Proposition~\ref{prop:3}, $\diam f_n (z) \to 0$ for all $z \in U$.
Now choose a point $x' \in f_n^{-1}(x_g) \cap \partial U$. Let $\mathcal{A}$ be the convex hull 
of $\{ x \} \cup \partial U$ and let $(z',x')$ be the edge of $\mathcal{A}$ with endpoint $x'$.
We proceed by contradiction and suppose that  $\deg_{x'} f_n$ is bounded.
By convexity (Lemma~\ref{lem:1}), it follows that $\deg_z f_n$ is bounded for all $z \in (z',x')$.
Thus, $d_\H (f_n (z), x_g)$ is also bounded for all $z \in (z',x')$. Since $\diam f_n (z) \to 0$,
we obtain the desired contradiction and conclude that $\deg_{x'} f_n \to \infty$.
\end{proof}

\begin{proof}[Proof of Proposition~\ref{prop:subfine2}]
 We rely on the following lemma whose proof is given thereafter.

Recall that in characteristic $p$, the Frobenius morphism is defined by $F (z) = z^p$.
When $p=0$, the Frobenius morphism is by convention the identity map.
\begin{lemma}\label{lem:bd-poss}
Let $S$ be any finite subset of $\P^1({\tilde{k}})$.
Then there exists a finite collection of separable rational functions $R_i$ with the property that 
any rational function $R$ such that $ R^{-1}\{ 0, 1,\infty \} \subset S$
is the composition of $R_i$ with some iterate of the Frobenius morphism.
\end{lemma}
Since $f_n (X)$ avoids $\{ 0,1, \infty\}$, and $f_n^{-1}(x_g) \cap X$ is reduced to $x_0$, any direction at $x_0$ which is not determined by a branch of $\cA_X$ is necessarily mapped to a direction avoiding $\{ 0,1,\infty\}$. In particular, we can 
apply the previous lemma to $T_{x_0} f_n$, and after taking a suitable subsequence we have
$$
T_{x_0} f_n = R\circ F^{d(n)}
$$
for some fixed separable fraction $R \in \tilde{k}(T)$, and some $d(n) \ge0$. Note that by our assumption $\deg_{x_0}(f_n) = \deg(R)\, p^{d(n)} \to \infty$ so that we can take $d(n)$ to be
strictly increasing to infinity. Note in particular that we have $\car (\tilde{k}) >0$.
\end{proof}

\begin{proof}[Proof of Lemma~\ref{lem:bd-poss}]
Since the triple $\{ 0, 1, \infty\}$ is totally invariant by the Frobenius morphism, we can write
$R = \tilde{R} \circ F^n$ for some $n$ where $\tilde{R}$ is separable and $\tilde{R}^{-1}\{ 0, 1,\infty \} \subset S$. 

First we show that $\deg(\tilde{R}) \le d - 2$ with $d:= \# S$.
Indeed we have $\# \tilde{R}^{-1} (0) = \deg(\tilde{R}) - \sum_{\tilde{R}^{-1}(0)}( \deg_{\tilde{R}}(x) -1)$, and similarly for $1$ and $\infty$.
Summing up we get
\begin{eqnarray*}
d \ge \# \tilde{R}^{-1} \{ 0, 1, \infty \}  &=& 3\deg( \tilde{R}) -  \sum_{\tilde{R}^{-1}\{0,1,\infty\}}( \deg_{\tilde{R}}(x) -1)\\
& \ge & 3\deg( \tilde{R}) - (2\deg( \tilde{R})-2)~,
\end{eqnarray*}
as required.
We can thus write $\tilde{R}$ under the form:
$$
\tilde{R} (T)= a \, \frac{\prod (T - z_i) } { \prod (T- z'_j)} 
$$
for some $a \in k$, and a collection of at most $d -2$ points  $z_i , z'_j \in S$, and such that $\tilde{R}(z'') =1$ for some $z'' \in S$.
The collection of such fractions $\tilde{R}$ is finite.
\end{proof}

%
%

\subsection{Proof of Theorem~B}\label{sec:proofs}

Recall the setting: $X$ is a connected open subset of $\poneank$, and $f_n : X \to \poneank \setminus \{ 0,1,\infty\}$ a sequence of analytic functions
pointwise converging to $f$ in $X$ such that $\unsepdeg_{x} (f_n)$ is bounded for all type II points $x \in X$. 

Since any point in a connected open subset of $\poneank$ has an affinoid neighborhood,
it is sufficient to consider an affinoid domain $Y\subset X$ and show that the restriction of $f$ to $Y$ is continuous.

\medskip
Let us first assume that $ \sup_n \sup_Y \deg_x (f_n) <  \infty$.
We claim that one can find a closed ball $B$ of positive diameter such that 
$f_{n}(Y) \subset \poneank \setminus B$ (possibly after extracting a subsequence).
Indeed if it is not the case, by the maximum principle we may find $x_0, x_\infty \in \partial Y$ such that
$f_n(x_0)$ converges in $\aoneank$ and $f_n(x_\infty) \to \infty$. On the other hand
we have
$$
d_\H (f_n(x_0) , f_n (x_\infty) ) \le \sup_n \sup_Y \deg_x (f_n) \times d_\H(x_0,x_\infty) < \infty~,
$$
which yields a contradiction.
We conclude that $f$ is continuous and $f( Y \cap \H) \subset \H$,  by Theorem~\ref{thr:pointwise}.

\medskip

When the local degree is unbounded, 
by extracting a subsequence  we can assume that $f_n^{-1}(x_g)$ is a fixed (finite) set $S$ of branched points of $\cA_X$.
Let us now assume that  $\sup_n \sup_Y \deg_x(f_n) = \infty$. 
By Corollary~\ref{cor:1}, one can find a point $x \in Y$ such that $\deg_x(f_n) \to \infty$.
By Proposition~\ref{prop:subfine}, after extraction either we are done, or 
we infer the existence of a branched point $x'$ of $\cA_Y$
such that $f_n(x') = x_g$ for all $n$, and $\deg_{x'} (f_n) \to \infty$. 
Hence, $\deg_{x'} (f_n) \to \infty$ for some $x' \in S$.

Now for all $x \in S$, consider the connected component $Y(x)$ of $\{ x \} \cup (X \setminus S)$ containing $x$. 
By Proposition~\ref{prop:subfine2} applied to $Y(x')$, we infer that $\unsepdeg_{x'} (f_n) \to \infty$. Theorem~B follows, since 
this is not possible by assumption.


\section{Examples}\label{sec:example}
We explore various examples of limits of analytic maps. 

\subsection{Limits of analytic maps}
Let us describe some typical maps appearing as limits of analytic functions avoiding three points in the projective line.

Consider a sequence $\zeta_n \in k$ such that $|\zeta_n| =1$, and $|\zeta_n - \zeta_m | =1$ for all $n\neq m$. Pick any integers $r \le s$, and $a \in k$.
Then the sequence $f_n: \poneank \to \poneank$ given by $f_n(z) =z^r + a \zeta_n z^s $ is pointwise converging to the unique continuous function whose restriction to the standard affine line sends a point $z \in k $ to the point corresponding to the ball $B(z^r , |a| \cdot |z|^{s})$.
Note that except in the trivial case $a=0$ this function is never analytic.

\subsection{Non continuous limits}
\label{sec:NotC0}
We now explore a  class of examples showing that  assumptions are needed to get continuous limits.

Pick any rational function $R \in k(T)$ of degree $d\ge 2$. Suppose all its coefficients are $\le 1$ in norm, and the reduction of $R$ in the residue field of $k$ has degree $d$ (in this case $R$ is said to have good reduction). Consider the sequence of analytic function $  f_n = R^{ n!}$.
Then $f_n$ converges pointwise to a function $f$ that is not continuous.

To see this, recall that $\poneank \setminus \{ x_g \}$ has a partition into open balls $B(\zeta)$
one per element $\zeta$ of the residue field $\tilde{k}$ of $k$. Denote by $\tilde{R}$ the residue map acting on $\P^1 ({\tilde{k}})$.

\begin{enumerate}
\item
If  $\zeta$ is not preperiodic for $\tilde{R}$, then $R^{n!}(x) \to x_g$ for any point $x \in B(\zeta)$.
\item
If $\zeta$ is preperiodic to a periodic cycle of $\tilde{R}$ that is critical, then 
one can find $c \in \poneank$ such that $R^{n!}(x) \to c$ for any point $x \in B(\zeta)$. 
\item
If $\zeta$ is preperiodic to a periodic cycle of $\tilde{R}$ that is non-critical, then 
$R^{n!}(B(\zeta))$ is eventually mapped to some open ball $B(\zeta')$ that is fixed by $R^N$ for some $N$. And pointwise convergent subsequences of $R^{N m}$ 
have  (continuous) non-constant limit maps on $B(\zeta')$.
\end{enumerate}
The description of the limit map in the last case highly depends on the characteristic of the field. 
When the characteristic of $k$ is zero, and the residual characteristic is positive, then the ball $B(\zeta')$ is a component of quasi-periodicity, and $R^{n!} \to \id$ on it.
Finally in positive characteristic, when $R$ is the Frobenius map, only cases (1) and (2) appear.

\medskip

In characteristic $p>0$, for $\e_n$ small enough,  the restriction of any polynomial $f_n = z^{p^n} + \e_n z^{p^n+1}$  to the affinoid domain $\{ |z| \le 2\} \cap \{  2 |z| \ge 1 \} \cap  \{ 2 |z -1 | \ge 1 \}$ avoids $\{ 0, 1, \infty\}$. By Theorem~A, one can extract many subsequences of $f_n$ that converge pointwise. However none of the obtained limits are continuous.

\subsection{Analytic maps avoiding fewer points}

Theorem~B does not hold with analytic maps avoiding only two points. 
Take $X = \aoneank \setminus \{ 0\}$, and $f_n(z) = z^n$. Any limit
is $0$ on  $\{ 0 < |z| <1\}$, and $\infty$ on $\{ 1 < |z| < \infty\}$.
Hence cannot be continuous at $x_g$.

\medskip

Theorem~A does not hold with analytic maps avoiding only one point.
Pick $c \in k$ such that $|c| \ge 4$ and consider the quadratic polynomial
$P_c(z) = z^2 + c$. Then $\{P^n_c\}$ is a family of analytic maps on $\aoneank$ with values 
in $\aoneank= \poneank \setminus \{ \infty \}$.
We claim that no subsequence of $P^n_c$ is converging pointwise on $\aoneank$.

 The filled-in Julia set $K_c = \{ z \in\aoneank, \, |P^n_c(z)| \text{ is bounded}\}$
 is a Cantor set included in $\A^1(k)$, and there exists a homeomorphism
 $\pi: K_c \to \{ 0,1 \}^\N$ such that $\pi \circ P_c = \sigma \circ \pi$
where $\sigma \{ \e_n \} = \{ \e_{n+1}\}$ is the left shift on $\{ 0,1 \}^\N$.

Suppose by contradiction that $P^{n_j}_c$ converges pointwise on $\aoneank$.
This would imply $\sigma^{n_j}$ to converge pointwise on $\{ 0,1 \}^\N$.
Choose any sequence $\e$ such that $\e_{n_j} =1$ if $j$ is odd, and $\e_{n_j} =0$ if $j$ is even. 
Then $\sigma^{n_j} (\e)$ does not converge, which gives a contradiction.

%
%
%
%

\section{Normal families and applications}\label{sec:normal}
Let us recall the  following notion from the introduction.
\begin{definition}
Let $X$ be any open subset of $\poneank$.

A family $\cF$ of meromorphic functions  on $X$ is normal if for any sequence $f_n \in \cF$
and any point $x \in X$, there exists a neighborhood $V \ni x$, and  a subsequence $f_{n_j}$ that is converging pointwise on $V$ to a continuous function.
\end{definition}


\subsection{Local conditions for normality}
In the complex case, Marty's theorem (e.g. see~\cite{marty}) is a characterization of the normality of a family of meromorphic maps in terms of the chordal derivative. Such a result is unclear in the non-archimedean context. However we prove
\begin{theorem}\label{thm:normal-char}
Suppose $X  \subset \aoneank$ is an open set containing $0$, and $\cF$ is a family of  analytic maps $f :  X \to \aoneank$ such that $\sup_\cF |f(0)| < + \infty$.

If the family  $\cF$ is  normal  in a neighborhood of $0$, then $\sup_\cF|f|$ and $\sup_\cF |f'|$ are both uniformly bounded in some neighborhood of $0$.  The converse statement holds if $\car(k) =0$.
\end{theorem}
As an immediate corollary, we infer
\begin{corollary}\label{cor:normal-fixed}
Suppose $f$ is an entire function fixing the origin.
Then $\{ f^n \}$ is a normal family in a neighborhood of $0$ if and only if
$|f'(0)| \le 1$.
\end{corollary}
The corollary also holds for rational maps. It is a consequence of Theorem~\ref{thm:normal} below and the fact that any repelling fixed point lies in the Julia set\footnote{A proof of this is given in~\cite[p.343]{baker-rumely:book} when $\car(k) =0$.}.
\begin{proof}
If $|f'(0)| \le 1$, then there exists a ball $B$ containing $0$ such that $f(B) \subset B$. By Theorem~\ref{thm:extract-affinoid}, the sequence $\{ f^n\}$ is a normal family on $B$.
The converse is a direct consequence of the previous theorem.
\end{proof}
\begin{proof}[Proof of Theorem~\ref{thm:normal-char}]
Assume the family is normal in some ball $B\ni 0$. 
Denote by $x_0 \in \H$ the point associated to $B$, and 
pick a subsequence $f_n$ (possibly with repetition) such that 
$\lim_n |f_n'(x_0)|  = \sup_\cF |f'(x_0)|$.
Reducing $B$ and passing to a subsequence,  we may suppose $f_n$ converges pointwise to a continuous function $g$ on $B$. Since $\sup |f_n(0)| < \infty$, and $g$ is continuous,  by extracting a further subsequence and rescaling the image, we may suppose $f_n(x_0)$ converges to a point in $B(0,1/2)$.  Therefore, $f_n(B) \subset \ov{B(0,1)}$ for $n\gg 1$, and  Schwarz' Lemma then implies $\sup_B |f_n'| = |f'_n(x_0)|$ is uniformly bounded. It follows that $\sup_{\cF} \sup_B  |f'| < \infty$ as required.

Assume conversely that $|f_n'| \le C$ in a some ball $B\ni 0$ of radius $r$, and look at the power series expansion $f_n(z) = f_n(0) + \sum_{i \ge 1} a^{(n)}_i z^i$.
We have $$\sup_{i \ge 1} |i|\, |a^{(n)}_i|\, r^i \le C.$$
If $\car(\tilde{k}) =0$, then $\sup_{i \ge 1} \, |a^{(n)}_i|\, r^i \le C$. Therefore, $f_n (B(0,r)) \subset B(0, C)$ and the claim follows from Theorem~\ref{thm:extract-affinoid}.
If $\car(\tilde{k}) >0$ and $0 < \rho < r/\car (\tilde{k})$, then
$$\sup_i  |a^{(n)}_i|\, \rho ^i \le \sup_i C |i|^{-1} \, \left( \frac{\rho}{r}\right)^i \le C \, \sup_i \left(\frac{\car(\tilde{k}) \, \rho}{r}\right)^i
< +\infty,$$ 
since $\car k =0$.
Thus,   for suitable positive constants $\rho, C'$ we have that $f_n (B(0,\rho)) \subset B(0, C')$.
Now we may apply  Theorem~\ref{thm:extract-affinoid} to establish that  $\cF$ is a normal in a neighborhood of $z=0$.
\end{proof}


\subsection{Normality and equicontinuity: proof of Theorem~C}

Suppose $\cF$ is a family of analytic functions on an open subset $X$ of $\poneank$
with values in $\poneank$.

Assume first that any $z \in \P^1(k)$ admits a neighborhood $U$ on which $\cF$ is a normal family. Let us split the family $\cF$ into two subfamilies $\cF_0 = \{ f \in \cF, \, |f(z)| \le 1 \}$, and $\cF_1 = \{ f \in \cF, \, |f(z)| >  1 \}$.
Then Theorem~\ref{thm:normal-char} applied to $\cF_0$ shows that $\sup_{\cF_0} \sup_U |f'| < \infty$. This implies the equicontinuity of $\cF_0$ on $U$. The same argument can be applied to $\cF_1$, since $1/f \in \cF_0$ for all
$f \in \cF_1$.

\medskip

Conversely, pick $z \in \P^1(k)$, and a ball $B$ containing $z$ such that 
$\{ f : B \to \P^1(k) \}_{f \in \cF}$ is equicontinuous. For  any sequence $f_{n}\subset \cF$, we must find a subsequence that is converging to a continuous function in a neighborhood $U \subset B$ of $z$. 
After extraction, and possibly replacing $f_n$ by $f_n^{-1}$, we may always assume that $|f_n(z)| \le 1$ for all $n$. By equicontinuity, shrinking $B$ if necessary, we conclude that 
$f_{n}(B)$ is included in the unit ball for all $n$.
Let $U$ be the convex hull of $B$. From Theorem~\ref{thm:extract-affinoid}, we conclude that there exists a subsequence  $f_{n_j}$ converging pointwise to a continuous function.


\subsection{Fatou set of a rational map}
Recall that the Julia set of a rational map $R$ of degree at least $2$ is the set of points $x \in \poneank$ such that for any open subset $U$ containing $x$, there exists an integer $n$ such that $R^n(U)$ contains all $\poneank$ but a countable set of discrete rigid points, see~\cite{baker-rumely:book,favre-letelier:equi}.

As a first application of our results, we prove
\begin{theorem}\label{thm:normal}
Suppose $R$ is a rational function of degree at least $2$. Then the Fatou set
of $R$ coincides with the set of points $x \in \poneank$ such that $\{ R^{n}\}$ forms a normal family in a neighborhood of $x$.
\end{theorem}
\begin{proof}
Suppose $x$ belongs to the Julia set $J(R)$ of $R$. 
Pick a subsequence $n_j$ such that $R^{n_j}(x)$ converges to a point $y \in \poneank$. 
Since $J(R)$ is closed, $y$ belongs to the Julia set.

The set of non-repelling rigid periodic points of $R$ is not empty.
In fact, the argument of Benedetto~\cite{benedetto} in characteristic $0$ extends verbatim in arbitrary characteristic as follows.
If $R$ admits a rigid periodic point with multiplier a root of unity, then this is clear.
Otherwise we can apply the Woods Hole formula, see~\cite{woodshole}, and~\cite[Corollaire 6.12]{SGA5} for a proof:
$$
1 = \sum_{R(p) =p} \frac1{1 - R'(p)}~.
$$
Now if $|R'(p)| >1$ for all $p$, the right hand side is $<1$ which gives a contradiction.

If there exists an attracting orbit, we let $V$ be a forward invariant union of closed balls containing  this cycle such that $V$ is contained in the Fatou set $F(R)$, and pick a point $y' \in V$ that is not periodic.
If there is an indifferent periodic cycle, we let $y'$ be one of the points of this orbit. For an appropriate
closed neighborhood $V \subset F(R)$ of the orbit we have that  $R(V) = V$.

In both cases, we have found a point $y'$ and a closed neighborhood $V \subset F(R)$ of $y'$ such that 
the cardinality of $R^{-n} \{ y'\}$ tends to infinity as $n\to \infty$
and $R(V) \subset V$. 

By~\cite{favre-letelier:equi}, the probability measures $d^{-n_j}R^{n_j*} \delta_{y'}$ converge to a measure whose support
is equal to $J(R)$. In particular, the closure of $\cup_j R^{-n_j} \{ y' \}$ contains $x$.
One can thus find a sequence $y_j$ tending to $x$ and such that $R^{n_j} (y_j) = y'$.
Now suppose $\{R^{n_j}\}$ is a normal family at $x$. Then we could find a (sub)-subsequence
$R^{n_{j_l}}$ that is converging pointwise to a continuous function $f$ on a neighborhood $U$ of $x$.
Pick $j$ large enough so that $y_j$ belongs to $U$. Then $R^{n_j}(y_j)  = y'$, hence 
$R^{n_{j'}}(y_j)  \in V$ for all $j' > j$. Thus, for all $j$ we would have that  $f(y_j) \in V \subset F(R)$, but $f(y_j) \to f(x)=y \in J(R)$.

\medskip
Now suppose $x$ belongs to the Fatou set. We claim that there exists a basic open subset $V \ni x$ such that $\cup_{n\ge 1} R^n(V)$ avoids a closed ball $B$. We give a proof of this fact thereafter and first conclude with the proof of the theorem.

Suppose first that there exists a point $x' \in V$ such that $\deg_{x'}(R^n) \to \infty$. Choose  coordinates such that $x_g, 0,1,\infty \in B$. Then Proposition~\ref{prop:subfine} can be applied and shows that any subsequence of $R^{n}$ admits a sub-subsequence that is converging to a constant. In particular the family $\{ R^n\}$ is normal in a neighborhood of $x$.

Next suppose that $\deg_{x'}(R^n)$ is bounded for all $x' \in V$. Then we apply Theorem~\ref{thm:extract-affinoid} with $X: = V$ and $Y : = \poneank \setminus B$. This shows
again that the family $\{ R^n\}$ is normal in a neighborhood of $x$.

We now indicate how to prove our claim. 
Let $U$ be the connected component of $F(R)$ that contains $x$. If $R^n(U) \neq U$ for any $n\ge 1$, then the family $\{ R^n\}_{n\ge 1}$ maps $U$ into $\poneank \setminus U$, and hence the iterates of any basic open set $V$  avoids a closed ball $B \subset U$. 
Thus we may restrict to the case in where $U$ is a periodic Fatou component, 
which we will assume to be fixed for sake of convenience. 
Suppose $U$ is the basin of attraction of a fixed point lying in $U$.
Then we pick $V$ to be a basic open set that is relatively compact in $U$ and contains $x$ and the fixed point in $U$. Then  $\cup_n R^n(V)$ is still relatively compact in $U$, hence its complement contains a closed ball.
Otherwise it is known that $U$ is a basic open set in which the local degree of $R$ is constant equal to $1$ and whose boundary points are type II Julia periodic points (see~\cite[Chapitre 5]{RiveraThesis} over $k= \C_p$, and~\cite[Proposition 2.16]{favre-letelier:equi} for a sketch in arbitrary complete fields).
Then $R$ fixes the skeleton of $U$ and permutes its boundary points, hence $R^N |\cA_U = \id$ for some $N$. Denote by $\pi(x)$ the unique point in $\cA_U$ such that $[\pi(x), x]\cap \cA_U = \{ \pi(x)\}$. Choose a small open (not necessarily connected) subtree $\cT$ of $\cA_U$ that is relatively compact in $U$, invariant by $R$, and such that  $\pi(x), R(\pi(x)), ... , R^{N-1}(\pi(x))$ belongs to $\cT$.  Set $V = \pi^{-1} (\cT)$. Then $V$ is $R$-invariant, and avoids $U \setminus \overline{V}$ that is a non empty open set.
\end{proof}


\section{Fatou-Julia theory of entire maps}\label{sec:entire}
We now explore the dynamics of  a transcendental entire (i.e. not a polynomial) map
$ f : \aoneank \to \aoneank$. 
In view of Theorem~\ref{thm:normal}, it is natural to make the following definition.
\begin{definition}
The Fatou set $F(f)$ of $f$ is the set of points $x$ where the sequence $\{ f^n \}$ forms a normal family in a neighborhood of $x$. The Julia set $J(f)$ is the complement of the Fatou set in $\aoneank$.
\end{definition}

Bezivin~\cite{bezivin:entire} studied the iteration of transcendental entire 
maps over $\C_p$  defining the Fatou set in terms of equicontinuity of the iterates with respect to the chordal metric. Theorem~C implies the intersection of our Fatou set with the set of rigid points is precisely the Fatou set in the sense of Bezivin.

%
%

\subsection{The Fatou set and the Julia set}

\begin{theorem}\label{thm:basics}
Let $f$ be any transcendental entire map of $\aoneank$. Then the following holds:
\begin{itemize}
\item
The Fatou set is  an open subset  of $\aoneank$ which is totally invariant under $f$.
\item
The Julia set is an unbounded closed totally invariant perfect subset of $\aoneank$.
Moreover, $J(f) \subset \overline{ \cup_{n \ge0} f^{-n} \{ z \}}$ for all $z \in \aoneank$.
\item
A periodic rigid orbit belongs to the Julia set if and only if  it is repelling.
\end{itemize}
\end{theorem}

The reader may find in~\cite[Propositions 5,6]{bezivin:entire} related results
concerning the rigid Julia set.

In order to prove the theorem it is convenient to establish the following.

\begin{lemma}
  \label{lem:5}
  Given a transcendental entire map $f$, let $\phi(\tau) = \sup_{|z| \le e^\tau} \{ \log |f(z)| \}$. Assume that $f(0) =0$. 
If $\phi$ is not locally affine at $\tau_0$ and $\phi(\tau_0) > \tau_0$, then $J(f) \cap \ov{\{z \in k, |z| = e^{\tau_0} \}} \neq \emptyset$.
\end{lemma}

\begin{proof}
  The function $ \phi(\tau) := \sup_{|z| \le e^\tau} \log |f(z)|$ is a piecewise affine and convex function on $\R$ with  slopes that are integral, non-zero,  and tending to infinity. 
Denote by $z_\rho$ the point associated to the ball $B(0,\rho)$ and observe
that $f(z_\rho) = z_{e^{\phi(\log \rho)}} $ if $\phi$ is not locally constant at $\rho$. Moreover the local degree of $f$ at $z_\rho$ is the slope of $\phi$ in a sufficiently small interval $(\log \rho, \varepsilon + \log \rho)$, $\varepsilon>0$. And the direction pointing to $0$ at $z_\rho$ maps onto the direction pointing to $0$ at $f(z_\rho)$ with degree given by the slope of 
$\phi$ in a sufficiently small interval $(-\varepsilon +\log \rho,  \log \rho)$.

Now consider $r=e^{\tau_0}$.
From what precedes, there exists a closed ball $B$ contained in a
direction at $z_r$ different than the direction of $0$ such that 
$f(B) = B(0,r).$
The set $K= \cap_{n \ge0} f^{-n} \ov{B}$ is a decreasing sequence of non-empty compact sets
hence is non-empty. Pick any point $x$ in the boundary of $K$. Then $x$ lies in the Julia set since $f^n(x)$ is bounded and there exists 
$x'\in (x, \infty)$ arbitrarily close to $x$ such that  
$f^n(x') = z_r$ for some $n$, hence $f^{n +m}(x') \to \infty$ as $m \to \infty$. It follows that  $J(f) \cap \ov{\{z \in k, |z| = r \}} \neq \emptyset$.
\end{proof}

\begin{proof}[Proof of Theorem~\ref{thm:basics}]
The first statement about the Fatou set is a consequence of the second on the Julia set.
The third statement is a direct consequence of Corollary~\ref{cor:normal-fixed}.

Without loss of generality, after a change of coordinates, we may assume that $f(0) =0$. By the previous lemma, one can find a point $x \in J(f)$.
Since $f$ is an entire function, $f^{-1}(x)$
is unbounded which shows $J(f)$ is unbounded.

Now we will simultaneously show that no point of  $J(f)$ is isolated and that for all $z \in \aoneank$ we have
that $J(f) \subset \overline{ \cup_{n \ge0} f^{-n} \{ z \}}$. In fact, given $z \in \aoneank$, let $Z = f^{-1}(z) \setminus \{ f^n(z) \mid n \ge 0 \}$.
Since $f$ is a transcendental entire function, $Z$ has infinite cardinality.
Note that $ z \notin f^{-n}(Z)$ for all $n \ge 0$.
Given $a \in J(f)$ and a  basic open set $X$ containing $a$ it is sufficient to show that $f^{-n}(Z) \cap X \neq \emptyset$ for some $n$.
By contradiction, suppose that $f^n(X)$ avoids $Z$. 
Thus $X$ has a finite number, say $\ell \ge 1$, of complementary balls. 
The number of complementary components of 
$f^n(X)$ are at most $\ell$, so for any subsequence of $f^{n_j}$, passing to a further subsequence $f^{n_{j_i}}$, we may assume that there are two distinct 
elements
$z_1$ and $z_2$ of $Z$ which are  in the same component of the complement
of $f^{n_{j_i}}(X)$ for all $i$. Thus $f^{n_{j_i}}(X)$ avoids a closed ball. From Theorem~\ref{thm:extract-affinoid} we obtain that $f^{n_{j_i}}$ has a subsequence converging to a 
continuous function, therefore $\{ f^n \}$ is normal in $X$ which contradicts the fact that $a \in J(f)$. 

Taking $z \in J(f)$ we conclude that $J(f)$ has no isolated point. Therefore, $J(f)$ is an unbounded, totally invariant perfect subset of $\aoneank$.
\end{proof}

%
%

\subsection{In residual characteristic zero}

\begin{proposition}
Suppose $\car(\tilde{k}) =0$, and pick  any non constant entire function on $\aoneank$.

Then the Fatou set is non empty; 
and the Julia set is included in the closure of the set of (rigid) periodic points. 
\end{proposition}
\begin{proof}
Since $\car(\tilde{k})=0$, then
$(k,|\cdot|)$ is not separable as a topological space, i.e., does not admit a countable dense subset. In particular, $\aoneank$ is  not separable.
Since by Theorem~\ref{thm:basics} $J(f)$ is the closure of a countable set, it follows that $F(f) \neq \emptyset$.

Pick any point $x \in \aoneank$ which is not in the closure of the set of rigid periodic points.
Then in some open neighborhood $U$ of a preimage by $f^2$ of $x$, the family of meromorphic functions
$$
g_n := \frac{f^n - \id}{f^n - f}  \cdot \frac{f^2 - f}{f^2 - \id}
$$
avoids the values $\{0,1,\infty\}$.
Since $\car(\tilde{k}) =0$,  Corollary~D implies $\{g_n\}$ is a normal family which shows $\{f^n\}$ is also normal in $U$. Whence $f^2(x)$ (and $x$) lie in the Fatou set.
\end{proof}

%
%

\subsection{The basin of infinity}
In the context of iterations of complex transcendental entire functions, a Baker domain is a periodic unbounded component of the Fatou set contained in the basin of infinity.
Our next result rules out the existence of Baker domains when the residual characteristic vanishes.

\begin{theorem}
  \label{lem:4}
Let $f$ be any transcendental entire map. 

Then the basin of attraction of infinity is connected. Moreover,  for any $x \in \aoneank$, there exist $y \in (x, \infty) \cap J(f)$ such that $\{f^n(y)\}_{n\in \N}$ is an increasing sequence converging to $\infty$.  In particular, the Julia set always contains non-rigid points and Fatou components are bounded. 
\end{theorem}

In the complex setting it is still unknown whether every connected component of the basin of infinity is unbounded.
 The second part of the above result is a non-archimedean analog of a result of Eremenko~\cite{eremenko}.

\begin{proof}
The fact that the basin of infinity is connected follows at once since any entire map preserves
the natural order on $\aoneank$. 

For the rest of the proof,  we may assume that $f(0) =0$. 
  Let $\phi$ be the associated function as in Lemma~\ref{lem:5}. 
  Given $x \in \aoneank$, we let $\rho>0$ be such that the closed ball $B(0,e^\rho)$ contains $x$ and $f(x)$. Consider $\rho_0 > \rho$ such that $\rho_1=\phi(\rho_0) > \rho_0$. For $n \ge 0$, the intervals $[\phi^n(\rho_0), \phi^{n+1}(\rho_0))$ are pairwise disjoint and cover $[\rho_0, \infty)$. Since there exists infinitely many 
$\tau$ such that $\phi$ is not locally affine at $\tau$ and $\tau > \rho_0$, we 
may consider a sequence $\tau_m \in [\rho_0, \rho_1)$ with the property that
for all $m \ge 0$ there exists $n_m\ge 1$ such that $\phi$ is not locally affine at
$\phi^{n_m}(\tau_m)$. Passing to a subsequence, we may assume that $\tau_m$ converges to $\tau' \in [\rho_0, \rho_1]$. 
Let $y$ be the point associated to the ball of diameter $e^{\tau'}$ and containing $x$.
We claim that $y$ lies in the Julia set of $f$.

If $\tau'$ is an accumulation point of the sequence $\tau_m$, then $y\in J(f)$ follows from Lemma~\ref{lem:5}. 
Otherwise  $\tau' = \tau_m$ for all $m$.
We may assume that $n_m$ is strictly increasing and pick a sequence of points $y_m$ in the segment $[x,\infty)$ such that  $y_m=f^{n_m}(y)$.
From Lemma~\ref{lem:5}, there exists a direction $c_m$ at $y_m$ which maps onto the direction of $0$ at $f(y_m)$ and that contains a point in the Julia set.
Let $z_m$ be a direction at $y$ which maps under $T_y f^{n_m}$ onto $c_m$. We claim that $z_\ell \neq z_m$ provided $\ell \neq m$. We may assume that $\ell > m$. It follows that
$T_y f^{n_m +1}(z_m)$ is the direction of $0$. Therefore, $T_y f^{n_\ell}(z_m)$ is also the direction of $0$ but $T_y f^{n_\ell}(z_\ell) = c_\ell$ which is
not the direction of $0$. Hence, $z_\ell \neq z_m$ if $\ell \neq m$ as required.

We conclude observing that  all directions $z_m$ contains Julia set elements,
so that $y$ admits infinitely many directions intersecting this set. Therefore, $y\in J(f)$.
\end{proof}

%
%

\subsection{Examples of entire maps}

\begin{example}
  Consider $\lambda \in k$ such that $|\lambda| > 1$. 
Let $a_1 =1$ and for all $n \ge 1$, 
$$a_{n+1} = \lambda^{-n} a_n.$$
Then
$$f(z) = \sum_{j \ge 1} a_j z^j$$
is a transcendental entire map such that for all $x \in \aoneank$, there exists 
$y \in [x,\infty)$ with $[y,\infty) \subset J(f)$.
\end{example}

Indeed, it is not difficult to check that for
$$n \log |\lambda| < \tau <  (n+1) \log |\lambda|$$
the corresponding function $\phi$ is affine and has slope
$n+1$, for all $n \ge 1$. Moreover, for all $\tau > \log|\lambda|$, we have that $\phi(\tau) > \tau$.
Pick any irrational $\tau_0 > \log|\lambda|$ 
and a small neighborhood $I$ of $\tau_0$ in $\R$. For a sufficiently large iterate, say $m$,  of the expanding map $\phi$, 
the interval $\phi^{m}(I)$ must contain a point of the evenly spaced points
of the form  $\{n \log |\lambda|\}_{n \in \N^*}$. Since at these points $\phi$ is not locally affine, from Lemma~\ref{lem:5} we conclude that an arbitrary neighborhood of the
point associated to the ball of radius $\exp(\tau_0)$ contains a Julia set element. Therefore, $[x_{|\lambda|}, \infty) \subset J(f)$ where $x_{|\lambda|}$
is the point associated to $B(0,|\lambda|)$. Since any arc $[x,\infty)$ coincides with $[x_{|\lambda|}, \infty)$ in a neighborhood of $\infty$, we conclude 
that $f$ has the desired property.

\bigskip
Baker~\cite{baker:wandering} constructed complex entire transcendental maps with multiply connected domains $U$ such that
$f^n(U) \neq f^m(U)$ for all $n \neq m$ and $f^n(U) \to \infty$. Such a domain $U$ is an example of  wandering Baker domains
(see also~\cite{kisaka-shishikura:wandering} for recent developments along this line). Our next example can be regarded as 
the non-archimedean analogue of Baker's examples.

\begin{example}
  Consider $\lambda \in k$ such that $|\lambda| > 1$. 
For $n \ge 5$ consider  a sequence of negative
integers $\ell_n$ such that $\ell_5 <0$,  $\ell_6 < 3 \ell_5$ and
$$\ell_{n+2} = (n+1) (\ell_{n+1} - \ell_{n})$$
for all $n \ge 5$.
Let $$ f(z) = \sum_{n \ge 5} \lambda^{\ell_n} z^n$$
and, for all $n \ge 5$, consider the open annulus $A_n$ obtained after removing the closed ball of radius $|\lambda|^{\ell_n - \ell_{n+1}}$ containing the origin from the open ball of radius $|\lambda|^{\ell_{n+1}- \ell_{n+2}}$ containing the origin.

Then $f$ defines a transcendental entire map such that  $f(A_n)=A_{n+1}$ and $A_n$ is a Fatou component contained in the basin of infinity, for all $n \ge 5$. 
\end{example}
In fact, after checking by induction that $\ell_{n+1} < (n-2) \ell_n$, it follows
that $0>\ell_{n+1} - \ell_n$ is strictly decreasing to $-\infty$. It is not difficult to conclude  that for
$$ |\lambda|^{\ell_n - \ell_{n+1}} \le r \le  |\lambda|^{\ell_{n+1} - \ell_{n+2}}$$
we have  $$\sup_{|z| \le r} |f(z)| = |\lambda|^{\ell_{n+1}} r^{n+1},$$
whenever $n \ge 5$. 
Thus, $\phi$ is not locally affine exactly at the sequence of points 
$\tau_n = {(\ell_n - \ell_{n+1})} \log |\lambda|$.
Moreover, $\phi(\tau_n) = \tau_{n+1}$ and $\phi(\tau_n,\tau_{n+1}) = (\tau_{n+1}, \tau_{n+2})$. Therefore, $f(A_n) = A_{n+1}$ and the annulus $A_n$ is contained in the basin of infinity. In particular, $A_n$ is contained in the Fatou set.
Let $x_n$ be the point associated to the ball of radius 
$|\lambda|^{\ell_n  - \ell_{n+1}}$.
It follows that $T_{x_n} f$ has degree $n+1$. By Lemma~\ref{lem:bd-poss}, 
we conclude that given a neighborhood $U$ of $x_n$ there exists $m$, such that $f^m(U)$ contains the closed ball associated to $x_{n+m}$.
From Lemma~\ref{lem:5}, $J(f) \cap U \neq \emptyset$. 
Thus, $x_n \in J(f)$ for all $n$. 

%
%

\subsection{Questions}
We end this article with some natural questions. 
\begin{qst}{\cite[Problem 1]{bezivin:entire}}
Does there exist a transcendental entire map whose Julia set admits no rigid point?
\end{qst}
A positive answer to this question would also give a positive answer to the next problem.
\begin{qst}{\cite[Problem 2]{bezivin:entire}}
Does there exist a transcendental entire map that admit no repelling rigid periodic points?\end{qst}
\begin{qst}{\cite[Problem 3]{bezivin:entire}}
Does there exist a transcendental entire map whose Julia set is equal to $\aoneank$?
\end{qst}
Recall that it is known that the Julia set of $\exp(z)$ is equal to $\C$. On the other hand any non-archimedean rational map admits at least one indifferent rigid fixed point, hence its Fatou set is never empty. The existence of indifferent rigid fixed points for a
transcendental entire map is however unclear.
\begin{qst}
Does there exist a transcendental entire map that admit no indifferent rigid periodic points?
\end{qst}
\begin{qst}
Does there exist a transcendental entire map having a Fatou component $U$ such that
$\cup_n f^n(U)$ is unbounded but $f^n|_U$ does not converge to $\infty$?
\end{qst}


\end{document}